\documentclass[11pt, a4paper]{article}

\usepackage{fullpage}
\usepackage{setspace}

\author{Omer Angel 
  \and Duncan Dauvergne
  \and Alexander E. Holroyd
  \and B\'alint Vir\'ag }
\title{The Local Limit of Random Sorting Networks}

\usepackage{etoolbox}
\usepackage[longnamesfirst]{natbib}
\usepackage{hyperref,url}
\newcounter{bibcount}
\makeatletter
\patchcmd{\@lbibitem}{\item[}{\item[\hfil\stepcounter{bibcount}{\thebibcount.}}{}{}
\renewcommand\NAT@bibsetup%
   [1]{\setlength{\leftmargin}{\bibhang}\setlength{\itemindent}{-\parindent}%
       \setlength{\itemsep}{\bibsep}\setlength{\parsep}{\z@}}
\makeatother
\usepackage{cleveref}
  \crefname{theorem}{Theorem}{Theorems}
  \crefname{customthm}{Theorem}{Theorems}
  \crefname{lemma}{Lemma}{Lemmas}
  \crefname{prop}{Proposition}{Propositions}
  \crefname{defn}{Definition}{Definitions}
  \crefname{corollary}{Corollary}{Corollaries}
  \crefname{section}{Section}{Sections}
  \crefname{figure}{Figure}{Figures}

\usepackage{geometry}
\usepackage{amsmath,amsfonts,amsthm,amssymb}

\numberwithin{equation}{section}

\usepackage{graphicx}
\usepackage{epstopdf}
\usepackage[utf8]{inputenc}
\usepackage[english]{babel}
\usepackage{natbib}

\bibpunct{[}{]}{,}{a}{}{;}
\renewcommand{\bibnumfmt}[1]{[\citeauthor \citeyear]}
\usepackage{enumitem}

\usepackage{bbm}
\theoremstyle{plain}
\newtheorem{theorem}{Theorem}[section]
\newtheorem{corollary}[theorem]{Corollary}
\newtheorem{prop}[theorem]{Proposition}
\newtheorem{lemma}[theorem]{Lemma}

\theoremstyle{definition}
\newtheorem*{definition}{Definition} 

\newtheorem*{remark}{Remark}
\newcommand{\N}{\mathbb N}
\newcommand{\Z}{\mathbb Z}
\newcommand{\E}{\mathbb E}
\newcommand{\R}{\mathbb R}
\newcommand{\HH}{\mathbb H}

\newcommand{\oper}{\operatorname}
\newcommand{\inte}{\mathbb{Z}}
\newcommand{\nat}{\mathbb{N}}
\newcommand{\real}{\mathbb{R}}
\newcommand{\half}{\mathbb{H}}

\newcommand{\prob}{\mathbb{P}}
\newcommand{\hookp}{\mathbf{P}}
\newcommand{\expt}{\mathbb{E}}
\newcommand{\indic}{\mathbbm{1}}

\newcommand{\sset}{\subset}

\newcommand{\la}{\lambda}
\newcommand{\al}{\alpha}

\newcommand{\mathforall}{\text{ for all }}

\newcommand{\mathand}{\;\text{and}\;}

\newcommand{\mathor}{\;\text{or}\;}
\newcommand{\mathas}{\;\text{as}\;}

\newcommand{\ga}{\gamma}
\newcommand{\ep}{\epsilon}

\newcommand{\sig}{\sigma}

\newcommand{\scrA}{\mathcal{A}}
\newcommand{\scrG}{\mathcal{G}}

\newcommand{\scrK}{\mathcal{K}}
\newcommand{\scrM}{\mathcal{M}}
\newcommand{\scrC}{\mathcal{C}}

\newcommand{\scrH}{\mathcal{H}}
\newcommand{\scrS}{\mathcal{S}}

\newcommand{\scrI}{\mathcal{I}}

\newcommand{\card}[1]{\left\vert #1 \right\vert}

\newcommand{\close}[1]{\mkern 1.5mu\overline{\mkern-1.5mu#1\mkern-1.5mu}\mkern 1.5mu}

\newcommand{\ddd}{\dots}

\newtheorem{innercustomthm}{Theorem}
\newenvironment{customthm}[1]
  {\renewcommand\theinnercustomthm{#1}\innercustomthm}
  {\endinnercustomthm}

\newcommand{\eqd}{\stackrel{d}{=}}

\newcommand{\cvgd}{\stackrel{d}{\to}}

\newcommand{\symdif}{\triangle}
\newcommand{\rev}{\text{rev}}
\newcommand{\id}{\text{id}}

\newcommand{\sleq}{\preceq}
\newcommand{\sgeq}{\succeq}
\newcommand{\smin}{\setminus}
\providecommand{\keywords}[1]{{\textit{Keywords:}} #1}
\newcommand{\note}[1]{}
\begin{document}
\maketitle

\begin{abstract}
   A sorting network is a geodesic path from $12 \cdots n$ to $n \cdots 21$
  in the Cayley graph of $S_n$ generated by adjacent transpositions.
  For a uniformly random sorting network, we establish the existence of a
  local limit of the process of space-time locations of transpositions in a neighbourhood of $an$ for $a\in[0,1]$ as $n\to\infty$.  Here time
  is scaled by a factor of $1/n$ and space is not scaled.

  The limit is a swap process $U$ on $\mathbb{Z}$.  We show that $U$ is
  stationary and mixing with respect to the spatial shift and has
  time-stationary increments.  Moreover, the only dependence on $a$ is
  through time scaling by a factor of $\sqrt{a(1-a)}$.

To establish the existence of $U$, we find a local limit for staircase-shaped Young tableaux. These Young tableaux are related to sorting networks through a bijection of Edelman and Greene.
  
%
%
%
\end{abstract}

\keywords{
Sorting network; random sorting network; reduced decomposition; Young tableau; local limit
}

\section{Introduction}

\label{S:intro}
Consider the Cayley graph of the symmetric group $S_n$ where the edges are given by adjacent transpositions
$\pi_i = (i, i+1)$ for $i \in \{1, \dots, n-1\}$.  The permutation
farthest from the identity $\id_n = 12\cdots n$ is the reverse permutation
$\rev_n = n\cdots 21$, at distance $\binom{n}{2}$. A \textbf{sorting network} is a path in this Cayley graph from the
identity to the reverse permutation of minimal possible length, namely
$N = \binom{n}{2}$.  Equivalently, a sorting network is a representation $\rev_n = \pi_{k_1} \pi_{k_2} \cdots \pi_{k_N}$, with the path being the sequence $\sig_t = \Pi_{i \le t} \pi_{k_i}$, so that $\sig_0 = \id_n$ and $\sig_N = \rev_n$.

For this reason, sorting networks are also known as {\bf reduced decompositions} of the reverse permutation. Under this name, the combinatorics of sorting networks have been studied in detail, and there are connections between sorting networks and Schubert calculus, quasisymmetric functions, zonotopal tilings of polygons, and aspects of representation theory. We refer the reader to \cite{stanley1984number, manivel2001symmetric, garsia2002saga, bjorner2006combinatorics} and \cite{tenner2006reduced}  for more background in this direction.

\begin{figure}
   \centering
   \includegraphics[scale=0.7]{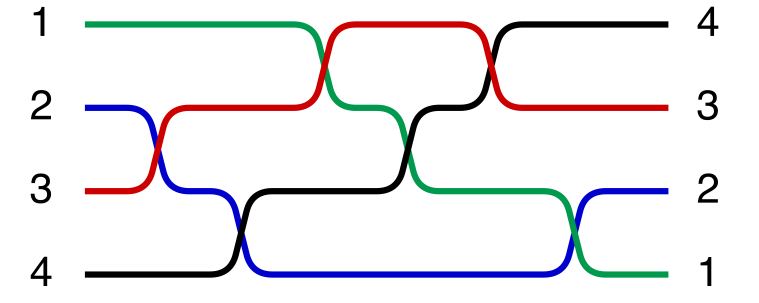}
   \caption{A ``wiring diagram" for a sorting network with $n = 4$. In
     this diagram, trajectories are drawn as continuous curves for
     clarity, whereas our definition specifies that trajectories make jumps at swap
     times.}
   \label{fig:wiring}
\end{figure}

Sorting networks also arise in computer science, as a sorting network can be viewed as an algorithm for sorting a list.  Consider an array with
$n$ elements, and let $\pi_{k_1} ,\pi_{k_{2}}, \dots, \pi_{k_N}$ be the
sequence of adjacent transpositions in a sorting network. At each step
$i$, instead of swapping the elements at positions $k_i$ and
$k_{i+1}$, rearrange these elements in increasing order.  After all $N$
steps, this process will sort the entire array from any initial
order. If we start with $\rev_n$, then every comparison will result in
a swap. 

It is helpful to think of the elements of $\{1,\dots,n\}$ as labeled
particles.  Each step in the sorting network has the effect of swapping
the locations of two adjacent particles.  In this way, we can talk of
the particles as having $\Z$-valued trajectories, with jumps of
$\{0,\pm 1\}$ at integer times. Exactly two particles make a non-zero
jump at each time.  We denote by $H_k(\cdot)$ the trajectory of
particle $k$. Specifically, for $t\leq N$ we have
$\sigma_{\lfloor t \rfloor}(H_k(t)) = k$ (here and later
$\lfloor t \rfloor$ denotes the integer part).

The number of sorting networks of order $n$ has been computed by
\cite{stanley1984number}.  Stanley observed that the number of sorting
networks equals the number of standard Young tableaux of a certain
staircase shape.  A bijective proof of this was provided by \cite{edelman1987balanced}.  Later, another bijective proof was
found by \cite{little2003combinatorial}, and recently \cite{HY} proved that the two bijections coincide.

\note{modify paragraph}
The study of random sorting networks was initiated by
\cite{angel2007random}.
That paper considered the possible scaling limits
of sorting networks, namely weak limits of the scaled process
\[
\lim_{n\to\infty} \frac{1}{n} H_{\lfloor an \rfloor }(t/N).
\]
Here, space is rescaled by a factor of $n$ and time by a factor of
$N=\binom{n}{2}$.  With this scaling, $H_{\lfloor an \rfloor}$ becomes a function from
$[0,1]$ to $[0,1]$, starting at $a$ and terminating at $1-a$.  It is
not a priori clear that the limit exists (in distribution) or even that the
limit is continuous.  While existence of the above limit is still an
open problem, it is shown in \cite{angel2007random} that the scaled
trajectories are equicontinuous in probability, and that subsequential
limits are
H\"older$(\alpha)$ for any $\alpha<1/2$.  

It is also conjectured --
based on strong numerical evidence -- that particle trajectories
converge to sine curves as $n \to \infty$.  We refer the reader to
\cite{angel2007random, angel2010random, kotowski2016} and \cite{rahman2016geometry}  for further
results and conjectures in this direction.
See also \cite{angel2009oriented} for the scaling limit of certain
non-uniform random sorting networks under this scaling.  Different local
properties of random sorting networks have also been studied in
\cite{angel2012pattern}.

\subsection{Limits of Sorting Networks}

In this paper we are interested in local limits of sorting networks.
These limits are local in the sense that space is not scaled at all.
However, time still needs to be scaled by a factor of $1/n$ to observe a
non-constant process.  Thus instead of the sorting process finishing at
time $N$, it will finish at time $N/n = (n-1)/2$.

\note{consistently use space-time, not time-space! Should look through -done}
\note{require no simultaneous jumps? do we prove this? Yes (injectivity).}

\begin{definition}
  A \textbf{swap function} is a function $U: \Z \times \R_+ \to \Z$ with
  the following properties:
  \begin{enumerate}[nosep,label=(\roman*)]
  \item For each $x$, we have that $U(x,\cdot)$ is cadlag.
  \item For each $t$ we have that $U(\cdot,t)$ is a permutation of $\Z$.
  \item Define the trajectory $H_x(t)$ by $U(H_x(t),t)=x$.  Then $H_x$
    is a cadlag path with nearest neighbour jumps for each $x$
    (i.e. the inverse permutation $U^{-1}$ is pointwise cadlag).
    \item For any time $t \in (0, \infty)$ and any $x \in \Z$, 
  $$
  \lim_{s \to t^-} U(x, s) = U(x + 1, t) \qquad \text{if and only if} \qquad  \lim_{s \to t^-} U(x +1, s) = U(x, t).
  $$
  \end{enumerate}
  \end{definition}
  We think of a swap function as a collection of particle trajectories $\{H_x(\cdot) : x \in \Z\}$.
Condition (iv) guarantees that the only way that a particle at position $x$ can move up at time $t$ is if the particle at position $x+1$ moves down. That is, particles move by swapping with their neighbours.

We let $\scrA$ be the space of swap functions endowed with the following topology. A sequence of swap functions $U_n \to U$ if each of the cadlag paths $U_n(x, \cdot) \to U(x, \cdot)$ and $H_{n, x}(\cdot) \to H_x(\cdot)$. Convergence of cadlag paths is convergence in the Skorokhod topology. We refer to a random swap function as a \textbf{swap process}.

\medskip

Our main result is the following limit theorem.

\begin{customthm}{1}\label{T:main}
  There exists a swap process $U$ so that the following holds.  Let
  $u\in(-1,1)$, and let $\{k_n : n \in \nat\}$ be any sequence such that $k_n/n\to (1+u)/2$.  Consider
  the shifted, and time scaled swap process
  \[
    U_n(x,t) = \sigma^n_{\lfloor nt/\sqrt{1-u^2} \rfloor}(k_n + x) - k_n,
  \]
  where $\sigma^n$ is a uniformly random $n$-element sorting network. Then
  \[
    U_n \xrightarrow[n\to\infty]{d} U.
  \]
  Moreover, $U$ is stationary and mixing of all orders with respect to
  the spatial shift, and has stationary increments in time: the
  permutation $(U(\cdot,s)^{-1}U(\cdot,s+t))_{t\ge 0}$ has the same law
  as $(U(\cdot,t))_{t\geq 0}$.
\end{customthm}

The scaling in Theorem \ref{T:main} can be thought of in the following way. We first choose a spatial location $u \in (-1 , 1)$ and look at a finite window around the position $(1 + u)n/2$. That is, we are concerned with particles whose labels are in a window $[(1+u)n/2 - K, (1 + u)n/2 + K]$. We want to know what the start of the sorting network looks like in this local window, at a scale where we see each of the individual swaps in the limit. To do this, we need to rescale time by a factor of $1/n$. Note that the semicircle factor of $\sqrt{1 - u^2}$ accounts for the fact that the swap rate is slower outside of the center of a random sorting network. On the global scale, this was proven in \cite{angel2007random}, so the slow-down does not come as a surprise.

To precisely define each $U_n$, for $x$ such that $k_n+x\notin \{1,\dots,n\}$, we
use the convention that $U_n(x,t)=x$. For $t>N/n$ we use the convention
that $U_n(x,t) = U_n(x,N/n)$.  By doing this, any sorting network corresponds to a swap function. Convergence in the above theorem is weak convergence in the topology on $\scrA$.

 Recall also that a process is spatially
mixing of order $m$ if translations by $k_1,\dots,k_m$ are
asymptotically independent as $\min |k_i-k_j| \to\infty$.  Spatial
mixing (even of order $2$) of the system implies ergodicity.



As a by-product of the proof, we also show that for any $t$,
there is a bi-infinite sequence of particles in the limit process $U$ that have
not moved by time $t$.  Consequently, $\Z$ can be split into finite
intervals that are preserved by the permutation $U(\cdot, t)$.
Furthermore, we prove convergence in expectation of the number of swaps
between positions $x$ and $x+1$ by some time $t$.  Specifically, if
$s(x, t, U)$ is the number of swaps between positions $x$ and $x+1$
up to time $t$ in the process $U$, then 
\[
\E s(x, t, U_n) \to \E s(x, t, U) = \frac{4}{\pi}t
\qquad    \text{as }n\to\infty.
\]
The expected number of swaps here agrees with corresponding global result obtained in \cite{angel2007random}.

Theorem \ref{T:main} is proven in the $k=0$ case as Theorem~\ref{k0}.  The
general case is a consequence of Theorem~\ref{lim-outside}.

\subsection{Limits of Young tableaux}

To prove Theorem \ref{T:main}, we will first prove a limit theorem for
staircase Young tableaux, and then use the Edelman-Greene bijection to
translate this into a theorem about sorting networks. This theorem is of
interest in its own right.

Recall that for an integer $N$, a \textbf{partition} $\lambda$ of $N$ is
a non-increasing sequence $(\lambda_1,\dots,\lambda_n)$ of positive
integers adding up to $N$.  The size of $\lambda$ is
$N = |\lambda| = \sum \lambda_i$.

We shall use the convention $\N=\{1,2,\dots\}$.  The \textbf{Young
  diagram} associated with $\lambda$ is the set $A\subset \N\times\N$
given by $A = \{(i,j) : j \leq \lambda_i\}$.  A Young diagram is
traditionally drawn with a square for each element, and elements of $A$
are referred to as squares.  The lattice $\N \times \N$ is usually
oriented so that the square $(1,1)$ is in the top left corner of the
lattice, but a different orientation will be convenient for us as
discussed below.
The \textbf{staircase} Young diagram of order $n$ is the diagram of the
partition $(n-1,n-2,\dots,1)$, of size $N=\binom{n}{2}$.  

A {\bf standard Young tableau} of shape $\lambda$ is an order-preserving
bijection $f: A\to \{1,\dots,N\}$, i.e., $f$ is increasing in both $i$
and $j$.  For both the statement of our results and their proofs, it
will be more convenient to work with \textbf{reverse standard Young
  tableaux}, where the bijection is order-reversing.  Clearly
$f \mapsto N+1-f$ is a bijection between standard and reverse standard
Young tableaux.

Our second main result is a limit theorem for the entries near the diagonal of a
uniformly random staircase shaped Young tableau of order $n$.  To
introduce this theorem, we must first change the coordinate system for
staircase Young tableaux.

\begin{figure}
   \centering
   \includegraphics[width=.5\textwidth]{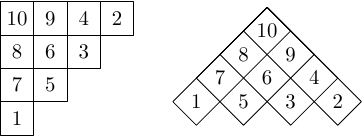}
   \caption{The staircase Young diagram of order $5$, i.e., of
     $\lambda = (4,3,2,1)$, with squares labelled by a reverse standard
     Young tableau, shown in both the usual and in our coordinate system
     in $\HH$.  One can think of the Young diagram $\lambda = (4,3,2,1)$
     as ten blocks in a triangular pile.  The entries in a tableau of
     shape $\lambda$ give a possible order in which to place these
     blocks in the pile while respecting gravity.}
   \label{fig:tableau}
\end{figure}

Define $\HH = \{(x,y) \in \Z\times\N : x+y \in 2\Z\}$.  We introduce a
partial order on $\HH$ given by $(x,y) \leq (x',y')$ if $x+y \leq x'+y'$
and $y-x \leq y'-x'$ (i.e.\ $(x,y)\leq (x',y')$ if there is a path in
the lattice from the $(x,y)$ to $(x',y')$, increasing in the
$y$-coordinate.  For $(c,n-1) \in \HH$, define 
$T(c,n) = \{z \in \HH : z \leq (c,n-1)\}$.
The set $T(c,n)$ is the image of a staircase shaped Young diagram of
order $n$ by the mapping $(i,j)\mapsto (c-i+j, n+1-i-j)$.
We extend the definition of
a staircase diagram of order $n$ and use that term for $T(c,n)$. 
We call the value $c$ the
\textbf{center} of the diagram.

The order on
$T(c,n)$ induced by the order on $\HH$ corresponds to reversing the
order on the Young diagram induced by the order on $\nat \times \nat$.  Therefore any order-preserving
bijection $G: T(c,n) \to \{1,\dots,N\}$ is a reverse standard Young
tableau.  We extend $G$ to a function from $\HH \to [0, \infty]$ by
setting $G(z) = \infty$ for all $z \notin T(c,n)$.  In the topology of
pointwise convergence in this function space, we then have the following
theorem about convergence of uniformly random reverse standard Young tableaux.

\begin{customthm}{2}
\label{T:YT_limit}
  There exists a random function $F:\HH\to[0,\infty)$ such that the
  following holds.  Fix $u\in(-1,1)$, and a sequence $k_n$ with
  $k_n/n\to u$.  Let $G_n$ be a uniformly random staircase
  Young tableau on $T(k_n,n)$. Then
  \[
    \frac{G_n}{n} \xrightarrow[n\to\infty]{d} \frac{1}{\sqrt{1 - u^2}} F.
  \]
  Moreover $F$ is stationary and mixing of all orders with respect to
  translations by $(2m,0)$ for $m \in \inte$.
\end{customthm}

The different components of Theorem
\ref{T:YT_limit} are proved in Theorem \ref{T:tableau-centre},
Proposition \ref{asym-ind}, and Theorem \ref{lim-outside} below.

\subsection*{Overview}

The structure of the paper is as follows.  Section \ref{S:hook-form} contains the
necessary background about Young tableaux and the Edelman--Greene
bijection, as well as some basic domination lemmas about Young tableaux.  This will allow us to conclude Theorem \ref{T:main} from the limit theorem for staircase Young tableaux, Theorem \ref{T:YT_limit}. Section
\ref{S:convergence} contains the proof of Theorem \ref{T:YT_limit} for the case $u=0$.
\note{reformulate is not quite right. revise overview - done...}

In order to translate Theorem \ref{T:YT_limit} using the Edelman--Greene
bijection to a theorem about sorting networks, we require certain
regularity properties of the Young tableau limit.  These are proved in
Sections \ref{S:rate} and \ref{S:regularity}.  Finally, we deduce Theorem \ref{T:main} in the case $u=0$ in
Section \ref{S:llcentre}.  In Section \ref{S:outside}, we extend Theorem \ref{T:YT_limit} and consequently
Theorem \ref{T:main} to arbitrary $u \in (-1, 1)$ by exploiting a monotonicity
property of random Young tableaux.

\bigskip

\noindent {\bf Remark.} \qquad We note that \cite{gorin2017} have results
that overlap some of ours.  Our proof of the local limit is
probabilistic, and is based on the Edelman--Greene bijection, the hook
formula and an associated growth process, and a monotonicity property
for random 
Young tableaux.  Gorin and Rahman take a very different approach,
using a contour integral formula for Gelfand--Tsetlin patterns discovered by 
\cite{petrov2014asymptotics}.  This allows them to get determinantal
formulas for the limiting process.  While for many models exact formulas
are the only known approach to limit theorems, we show that for random
Young tableaux the local limit and its properties can also be
established from first principles.

\section{The Hook Formula, the Edelman--Greene Bijection and Tableau Processes}
\label{S:hook-form}
In this section, we introduce some preliminary information regarding
Young tableaux and the Edelman--Greene bijection. We then use the hook formula to prove some basic domination lemmas about pairs of growing tableau processes.
\paragraph{The hook formula.}

Let $d(\lambda)$ be the number of reverse standard Young tableaux of shape
$\lambda$. \cite{frame1954hook} proved
a remarkable formula for $d(\lambda)$.  To state it, we first need
some definitions.  Let $A(\la) \subset \N \times \N$ be the Young diagram of
shape $\lambda$. For a square $z = (i,j) \in A(\la)$, define the {\bf hook}
of $z$ by
\[
  H_z = \Big\{(i,j') \in A : j' \geq j\Big\} \cup \Big\{(i',j) \in A: i'
  \geq i\Big\}.
\]
Define the \textbf{hook length} of $z$ by $h_z = |H_z|$.  We also define
the {\bf reverse hook} for $z$ by
\[
  R_z = \{w\in A\smin \{z\} : z \in H_w\}.
\]
The reverse hook will be of use later when manipulating the hook formula. We note here for future use hook lengths and reverse hook lengths in $\HH$. For a point $z = (x, y)$ in a diagram $T(c, n)$, we have that $h_z = 2y - 1$, and that $|R_z| = n - 1 - y$.

\begin{theorem}[Hook Formula, \cite{frame1954hook}]
  With the above notations, we have
  \[
    d(\lambda) = \frac{|\lambda|!}{\prod_{z \in A(\lambda)} h_z}.
  \]
\end{theorem}

\paragraph{The Edelman--Greene bijection.}

For the staircase Young diagram of order $n$, the hook formula gives
\[
  d(\lambda_n) = \frac{\binom{n}{2}!}{1^n 3^{n-1} 5^{n-2} \ddd (2n -3)^1}.
\]
As noted, this is also the formula for the number of  sorting networks
of order $n$ given in \cite{stanley1984number}.  We now describe the
bijection between these two sets given by \cite{edelman1987balanced}.

We recount here a version of the Edelman--Greene bijection for rotated
(defined on subsets of $\HH$) reverse standard Young tableaux.  More
precisely, the map as we describe it gives a bijection between Young
tableaux on the diagram $T(c,n)$ and sorting networks of size $n$ with
particles located at positions $\{c-(n-1), c-(n-1)+2, \dots, c+(n-1)\}$.
Note that here particles are located at positions in $2\Z$, and not in $\Z$ as in the statement of Theorem \ref{T:main}.
This is done to optimize the description of the bijection.  To
accommodate this, for odd $k$ we use $\pi_k$ to denote the swap of the particles at positions
$k-1$ and $k+1$.
\note{what is a sorting network on an interval? I've removed  this language.}

Given a reverse standard Young tableau $G: T(c,n) \to \{1,\dots,N\}$, we
generate a sorting network $\pi_{k_1} \pi_{k_2}, \ddd \pi_{k_N}$ and a sequence of Young
tableaux $(G_t)_{t\leq N}$, starting with $G_0=G$.  Recall that by
convention $G(z)=\infty$ for $z\notin T(c,n)$.  We repeat the following for
$t\in\{1,\dots,N\}$, computing $k_t$ and $G_t$ from
$G_{t-1}$.  (See \cref{fig:EG1,fig:EG2} for an example.)


\begin{description}
  
\item[Step 1:] Find the point $z_*\in\HH$ such that the value of
  $G_{t-1}(z_*)$ is minimal.  Clearly $z_*= (k,1)$ for some odd $k$.  Set $k_t = k$.

\item[Step 2:] Recursively compute the ``sliding path'' $z_1,z_2,\dots$ as
  follows. Set $z_1=z_*$.  If $z_i = (x,i)$, then
  $z_{i+1} \in (x\pm 1, i+1)$ is chosen to be the point with a smaller value of
  $G_{t-1}$.  If both are infinite then the choice is immaterial.

\item[Step 3:] Perform sliding to update $G$: If $z$ is in the sliding
  path, so that $z=z_i$ for some $i$ then let
  $G_t(z) = G_{t-1}(z_{i+1})$.  Otherwise, let $G_t(z) = G_{t-1}(z)$.
\end{description}

\begin{figure}
  \centering
  \includegraphics[width=.9\textwidth]{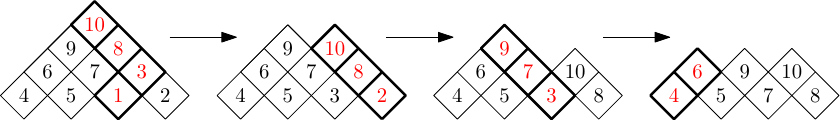}
  \caption{The first three iterations in the Edelman--Greene bijection.
    Squares not shown have $G_t(z)=\infty$.  In each iteration, (the start of) the sliding path is in bold.}
  \label{fig:EG1}
\end{figure}

\begin{figure}
  \centering
  \includegraphics[width=.9\textwidth]{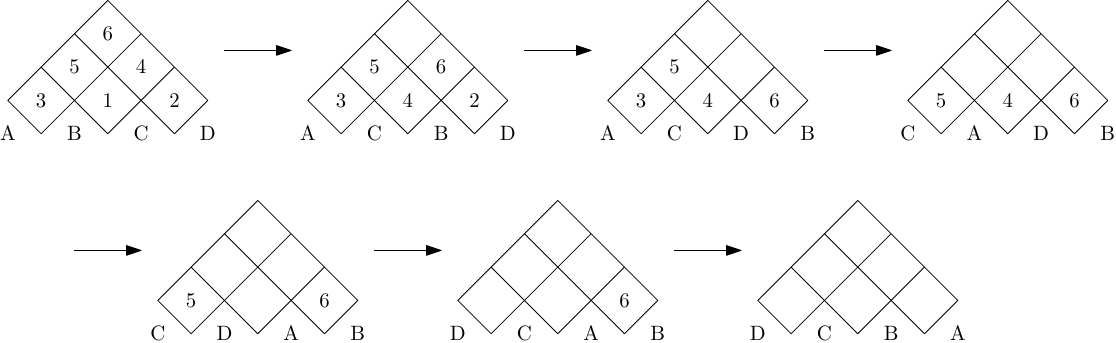}
  \caption{The Edelman--Greene bijection applied to a tableau of order
    $n=4$.  The particles are labelled A--D to distinguish them from the
    entries in the tableau.  The sorting network corresponds to the
    wiring diagram shown in \cref{fig:wiring}.}
  \label{fig:EG2}
\end{figure}

The output of the Edelman--Greene bijection is the swap sequence $(k_i)$
of length $N=\binom{n}{2}$, taking odd values
$k_i \in [c-(n-2),c+(n-2)]$.  Edelman and Greene proved that applying
the given sequence of swaps will reverse the elements of the interval
$[c-(n-1),c+(n-1)] \subset 2\Z$, and moreover, that any sorting network
on this interval results from a unique reverse standard Young tableau on
$T(c,n)$.


\subsection{Uniform Young tableaux}


The Edelman--Greene bijection allows us to sample a uniformly random
sorting network of size $n$ given a uniformly random reverse standard Young tableau of
shape $T(c,n)$.  We say a set $A\subset\HH$ is downward closed if whenever $z \in A$ and
$w\leq z$, then $w\in A$.  In the language of Young diagrams,
such an $A$ is a special case of a skew Young diagram.  Given a reverse
standard Young tableau $G$ on $T(c,n)$, let $A_i = \{z : G(z) \leq i\}$.
Monotonicity of $G$ implies that $A_i$ is downward closed.  Moreover
$|A_i| = i$ for each $i\in\{0,\dots,N\}$, and we have
$A_i \subset A_{i+1}$.  Thus a Young tableau on $T(c,n)$ can be viewed
as a maximal sequence of downward closed subsets
$A_0 = \emptyset \subset A_1 \subset A_2 \subset \dots \subset A_N =
T(c,n)$.  The complementary sets $B_i=T(c, n) \setminus A_i$ are rotated Young
diagrams, and $G|_{B_i}$ is a reverse standard Young tableau on that
diagram (with entries shifted by $i$).
 
If $B$ is a Young diagram, and $G$ is a reverse standard Young tableau
on $B$, then $G(z)=1$ for some square $z\in B$, and this square must
have hook $H_z = \{z\}$.  We call such squares {\bf corners} of $B$.
The restriction of $G$ to $B\setminus\{z\}$ is a reverse standard Young
tableau with all values increased by $1$.  This observation allows us to
use the hook formula to find the probability that in a uniformly random
reverse standard Young tableau of shape $\lambda$, the square containing
1 is a given corner $z$. We call this the {\bf hook probability}, denoted
$\hookp(B,z)$.  A simple calculation shows that
\[
  \hookp(B,z) = \frac{d(B\setminus \{z\})}{d(B)}
  = \frac{1}{|B|} \prod_{y \in R_z} \left(\frac{h_y}{h_y - 1} \right).
\]

This gives a simple procedure for sampling a uniformly random reverse
standard Young tableau on any diagram $B$: Pick a random corner $z_1$ of
$B$ with probability mass function $\hookp(B,z)$ and set $G(z_1)=1$.
Recursively pick a corner $z_2$ of $B\setminus\{z\}$
and set $G(z_2) = 2$, and repeat until all elements of $B$ have been chosen.  In terms of the corresponding growing
sequence of sub-diagrams, this takes the following form: Set
$A_0 = \emptyset$.  Having chosen $\{A_0,\dots,A_{i-1}\}$, pick a corner
$z_i$ of $B\setminus A_{i-1}$ with probability mass function
$\hookp(B\setminus A_{i-1}, z_i)$, and let $G(z_i)=i$ and
$A_i = A_{i-1} \cup \{z_i\}$.  We will primarily be interested in this
process when $B$ is a staircase diagram $T(c,n)$.

\begin{remark}
  While the hook probabilities have an explicit formula, which we use
  directly, one can sample a corner of a diagram with this distribution
  very efficiently using the \emph{hook walk}, a process described in
  \cite{greene2014young}.  We omit the mechanism of the walk since we do
  not need it, but remark that it can be used to provide alternate
  proofs of some of the stochastic domination lemmas that follow.
\end{remark}

\subsection{Continuous time growth}

A significant simplification of our analysis is achieved by Poissonizing
time.  Instead of generating a sequence of growing diagrams $A_i$, we
shall define a continuous time process with the same jump distribution
but moving at the times of a Poisson process.

The \textbf{staircase tableau process} (or simply tableau process) is a
Markov process $X(t) = X(c,n,r)(t)$.  Its law is determined by
parameters $c,n$ and $r$, and it is related to the uniform reverse
standard Young tableau of $T(c,n)$.  The state space of this process
comprises all downward closed subsets $A \subset T(c,n)$.  The initial
state is $X(0) = \emptyset$.  If $A$ and $A\cup\{z\}$ are two states,
then the rate of jump from $A$ to $A\cup\{z\}$ is 
\begin{equation}
  \label{rate-def}
  v_X(z, A) = r \cdot \hookp(T(c, n) \setminus A, z).
\end{equation}
When the process $X$ is clear from context we omit the subscript on the rate $v$. No other jumps are possible.  

\medskip 

Note that the parameter $r$ simply
multiplies all jump rates, so that the process $X(n,c,r)(t)$ has the
same law as $X(n,c,1)(rt)$.  Running these processes at different rates
will be useful, hence the inclusion of $r$ in the notations.  The state
$T(c,n)$ is absorbing. The total rate of jumps from any other state is
$r$, so the first ${n \choose 2}$ jump times of the process coincide with points in a rate-$r$ Poisson process.

Given the process $X$, let the inclusion time of a square $z$ be defined by
\[
  F(z) = \inf \{t : z \in X(t)\}.
\]
These determine the process $X$, since we have $X(t) = \{z : F(z) \leq
t\}$.  Note that $F$ is naturally defined on all
of $\HH$, with $F(z)=\infty$ for $z\notin T_{c,n}$, so that $F \in
[0,\infty]^\HH$. We refer to $F$ as the {\bf inclusion function} for $X$.
The first convergence theorem we prove can now be stated.

\begin{theorem} 
  \label{T:tableau-centre} 
  Let $X_n = X_n(c_n,n,n)$ be a sequence of tableau processes with
  $c_n=o(n)$, and let $F_n$ be the corresponding sequence of inclusion
  functions.  Then $F_n \xrightarrow[n\to\infty]{d} F$ for some random
  $F:\HH\to\R_+$.  Moreover, the limit $F$ is translation invariant, in
  the sense that $F \eqd F  \tau$, where $\tau (x,y) = (x+2, y)$.
\end{theorem}

We will use the notation $\tau$ throughout the paper to signify horizontal translation on $\half$.

\medskip
\note{Fix the following: done}

By the law of large numbers for the Poisson process, the limit of the inclusion functions $F_n$ for the processes $X_n$ is the same as the limit of a uniformly random reverse standard Young tableau on $T(c_n, n)$ with entries scaled by $1/n$. Thus Theorem \ref{T:tableau-centre} immediately implies the convergence and translation invariance in Theorem \ref{T:YT_limit} in the case $u = 0$. We will similarly prove Theorem \ref{T:YT_limit} for $u \neq 0$ in Section \ref{S:outside} by again Poissonizing time, noting that this does not change the limit.

\medskip

Note also that $T(c, n)$ is only defined when $c$ and $n$ have opposite parity, so when taking tableau limits for constant $c$, we may need to change the value of $c$ by 1, depending on whether $n$ is odd or even. In all of our proofs, shifting the position that the tableaux are centered at by 1 does not affect any of the arguments, as all of the domination lemmas we use are unaffected by distance changes of size $o(n)$. Therefore from now on, we will ignore issues of the parity of $c$ and $n$.

\subsection{Stochastic domination}
\label{S:dom}
A central tool in our proof of existence of certain limits is stochastic
domination of growth processes.  Subsets of $\HH$ are naturally ordered
by inclusion.  For coupled tableau processes $X$ and $X'$, we say that $X$ is
\textbf{dominated by $X'$ up to time} $T$ if for all $t\leq T$
we have $X(t) \subset X'(t)$.  In terms of the inclusion functions, this
can be stated equivalently as $F \geq F' \wedge T$ in the pointwise order on inclusion functions (note the order
reversal: a smaller process $X$ corresponds to larger inclusion times
$F$.)  In light of Strassen's theorem (see \cite{strassen}), we have that $X$
is stochastically dominated by $X'$ if there is a coupling of the two so
that domination holds, and write $X \sleq X'$ up to time $T$.


The next lemma gives a sufficient condition for stochastic domination of
one tableau process by another, in terms of their rates.

\begin{lemma}
  \label{stoch-dom-criterion}
  Let $X_1$ and $X_2$ be two tableau processes on the diagrams $T(c_1, n_1)$ and $T(c_2, n_2)$ respectively. Let $\scrS$ be some
  subset of the state space of $X$, and let the stopping time $T$ be the
  first time $t$ that $X_1(t) \notin \scrS$.  Suppose that for any
  $A_1 \in \scrS$, any state $A_2$ with $A_1 \subset A_2$, and for any lattice
  point $z$ we have
  \begin{equation}
    \label{rate-leq}
    v_{X_1}(z, A_1) \leq v_{X_2}(z, A_2),
  \end{equation}
  provided both are non-zero.  Then $X_1 \sleq X_2$ up to time $T$.
\end{lemma}


\begin{proof}
  Suppose first that $\scrS$ is the entire state space of $X_1$. The proof when $\scrS$ is not the whole state space goes through in the same way. 
  
  We define a Markov process $Y$ whose state space is all pairs
  $(A_1, A_2)$ with $A_1 \subset A_2$ such that $Y$
  has marginals $X_1$ and $X_2$.  We define the transitions rates of $Y$
  out of a state $(A_1, A_2)$ as follows.  Let $C_1$ be the set of all
  corners of $T(c_1, n_1)\setminus A_1$, $C_2$ be the set of all corners
  of $T(c_2, n_2)\setminus A_2$, and $C_1'$ be the set of all corners
  belonging to both $T(c_1, n_1)\setminus A_1$ and
  $T(c_2, n_2)\setminus A_1$. 
  
  For $z \in C_1'$, $Y$ transitions to state
  $(A_1 \cup \{z\}, A_2 \cup \{z\})$ with rate $v_{X_1}(z, A_1)$. $Y$
  also transitions to state $(A_1, A_2 \cup \{z\})$ with rate
  $v_{X_2}(z, A_2) - v_{X_1}(z, A_1)$. For $z \in C_2 \smin C_1'$, $Y$
  transitions to state $(A_1 , A_2 \cup \{z\})$ with rate
  $v_{X_2}(z, A_1) $.  For $z \in C_1 \smin C_1'$, $Y$ transitions to
  state $(A_1 \cup \{z\}, A_2)$ with rate $v_{X_1}(z, A_1) $.

  It is easy to check that $Y$ has the correct marginals and provides a
  coupling of $X_1$ and $X_2$ with $X_1 \leq X_2$.
\end{proof}

We can further simplify which rates we need to compare to prove stochastic domination with the following observation.

\begin{lemma}
\label{stoch-dom-1}
Suppose $X(t) = X(c, n, r)(t)$ is a tableau process, and  $A_1 \subset A_2$. Then $v(z, A_1) \leq v(z, A_2)$ for any point $z \notin A_2$.
\end{lemma}

\begin{proof}
The only interesting case here is when $z$ is a corner for both $T(c, n) \setminus A_1$ and $T(c, n) \setminus A_2$. Then by Equation \eqref{rate-def} and the hook probability formula,

\begin{equation*}
v(z, A_i)= \frac{r}{\card{T(c, n) \setminus A_i}}\prod_{y \in R^i_z} \left(1 + \frac{1}{h^i_y -1} \right).
\end{equation*}

Here $R^i_z$ refers to the reverse hook for $z$ in the diagram $T(c, n) \setminus A_i$, and $h^i_y$ refers to the cardinality of the hook for $y$ in the same diagram. We have that $R^1_z = R^2_z$, and each of these are simply the reverse hook for $z$ in $T(c, n)$. Also, $h^2_y \leq h^1_y$ for all $y$ since $T(c, n) \setminus A_2 \sset T(c, n) \setminus A_1$, and $\card{T(c, n) \setminus A_2} \leq \card{T(c, n) \setminus A_1}$. Putting this together, we get that $v(z, A_1) \leq v(z, A_2)$, as desired.
\end{proof}

To prove the more general stochastic domination result, we need the following lemma to help bound hook probabilities.

\begin{lemma}
\label{hook-prob-inequality}
Let $a, b$ be either two integers greater than $1$ or two half-integers greater than $1$, and define
$$
y= \prod_{i=a}^b \left(1 + \frac{1}{2i -1} \right) = \prod_{i=a}^b \frac{2i}{2i -1},
$$
where the product runs over integers between $a$ and $b$ if both are integers, and over half-integers between $a$ and $b$ if both are half-integers.
Then
$$\sqrt{\frac{2b-1}{2a-1}} < y < \sqrt{\frac{2b}{2a-2}}.
$$
\end{lemma}

\begin{proof}
We have $x < y < z$ where
$$
x=\prod_{i=a}^b \frac{2i+1}{2i} ,\qquad  z =\prod_{i=a}^b
\frac{2i-1}{2i-2}.
$$
Then $xy$ and $yz$ are telescoping products given by
$$
xy=\frac{2b -1}{2a-1}, \qquad yz=\frac{2b}{2a-2}
$$
so we get $\sqrt{xy}< y <\sqrt{yz}$.
\end{proof}

Now we can prove the following more general lemma about stochastic domination.

\begin{lemma}\label{stoch-dom-general}
Let $T(c_1,n_1) \sset T(c_2, n_2)$, and consider two tableau
processes
\[
  X_1 = X_1(c_1, n_1, n_1), \qquad \text{and} \qquad X_2 =
  X_2(c_2, n_2, \theta n_2).
\]
Fix $\al \in (0,1)$, and let $T$ be the stopping time when
$\lfloor \al \binom{n_1}{2} \rfloor$ lattice points have been added to
$X_1$.  Let the difference between the horizontal centers of the two
tableau processes be $d= |c_1 -c_2|$. Then $X_1 \sleq X_2$ up to time $T$,
provided that

$$
\theta > \frac{(n_2 -1)n_1}{(n_1 - 1)(n_2 - 2)}\left((1- \al)\sqrt{1- \left(\frac{n_1 + d}{n_2-2}\right)^2}\right)^{-1}.
$$
\end{lemma}

\begin{proof}
We may assume that $c_1 = d$ and $c_2 = 0$. By Lemmas \ref{stoch-dom-criterion} and \ref{stoch-dom-1} it suffices to show that for any state $A$ with $\card{A} \leq \lfloor \al {n_1 \choose 2} \rfloor$, and any corner $z$ of both $T(d, n_1) \setminus A$ and $T(0, n_2) \setminus A$ we have  
$$
v_{X_1}(z, A) \leq v_{X_2}(z, A).
$$
Let $R_z^1$ be the reverse hook of $z$ in $T(d, n_1)$, and let $h_y^1$ be the hook length of $y$ in $T(d, n_1)\setminus A$, and similarly define $R_z^2$ and $h_y^2$ for $T(0, n_2)$.
To get a simple expression for $\frac{v_{X_2}(z, A)}{v_{X_1}(z, A)}$, observe that if $y \in R_z^1$, then $y \in R_z^2$ and $h_y^1 = h^2_y$ for such $y$.  Thus

\begin{equation}
\label{eqn-1-stoch-dom}
\frac{v_{X_2}(z, A)}{v_{X_1}(z, A)} = \theta \frac{n_2}{n_1} \frac{\card{T(d, n_1) \setminus A}}{{\card{T(0, n_2) \setminus A}}} \prod_{y \in R_z^2 \setminus R_z^1}\left(1 + \frac{1}{h_y^2 -1} \right).
\end{equation}

We will show that this is always greater than 1. For $y = (y_1, y_2)$ to be in $R_z^2 \setminus R_z^1$ where $z =(z_1, z_2)$, one of two possibilities must occur. Either
\begin{align*}
(y_1, y_2) \in E_1 &= \{ (z_1- i , z_2 + i) : z_2 - z_1 + 2i \in (n_1 - 1 - d, \;n_2 -1]\}, \mathor \\
(y_1, y_2) \in E_2 &= \{ (z_1+ i , z_2 + i) : z_1 + z_2 + 2i \in (n_1 - 1+ d,\; n_2 -1]\}.
\end{align*}
For $(y_1, y_2) = (z_1- i , z_2 + i) \in E_1$ and for $y =(z_1+ i , z_2 + i)\in E_2$, the hook for $y$ is of length $1 + (y_2 -z_2) + (y_2 - 1) = 2i + z_2$. Thus using Lemma \ref{hook-prob-inequality}, we find that

\begin{align*}
 \prod_{y \in R_z^2 \setminus R_z^1}\left(1 + \frac{1}{h_y^2 -1} \right) &= \\
 \prod_{i = [n_1 - 1 - d - z_2 + z_1]/2 + 1}^{[n_2 -1 - z_2 + z_1]/2}&\left(1 + \frac{1}{2i + z_2 - 1}\right)\prod_{i = [n_1 - 1 + d - z_2 - z_1]/2 + 1}^{[n_2 - 1 - z_2 - z_1)]/2}\left(1 + \frac{1}{2i + z_2 - 1}\right) \\
& >  \sqrt{\frac{(n_2 -2)^2 - z_1^2}{n_1^2 - (z_1 - d)^2}}.
\end{align*}
Thus for the quantity \eqref{eqn-1-stoch-dom} to be greater than 1, we need

\begin{equation}
\label{eqn-stoch-dom-2}
\theta  \ge \frac{n_1}{n_2} \frac{\card{T(0, n_2) \setminus A}}{{\card{T(d, n_1) \setminus A}}}\sqrt{\frac{n_1^2 - (z_1 - d)^2}{(n_2 -2)^2 - z_1^2}},
\end{equation}
for all values of $z_1$ and $A$ with $\card{A} \leq \lfloor \al {n_1 \choose 2} \rfloor$. We then have the following chain of inequalities for the right hand side of \eqref{eqn-stoch-dom-2}, which show that the inequality \eqref{eqn-stoch-dom-2} holds for the values of $\theta$ specified in the Lemma.

\begin{align*}
\frac{n_1}{n_2} \frac{\card{T(0, n_2) \setminus A}}{{\card{T(d, n_1) \setminus A}}}\sqrt{\frac{n_1^2 - (z_1 - d)^2}{(n_2 -2)^2 - z_1^2}} &< \frac{n_1}{n_2} \frac{{n_2 \choose 2}}{(1-\al){n_1 \choose 2}}\frac{n_1}{n_2 -2}\left(1 - \left(\frac{z_1}{n_2-2}\right)^2\right)^{-1/2} \\
&\leq \frac{(n_2 -1)n_1}{(n_1 - 1)(n_2 -2)}\left((1- \al)\sqrt{1- \left(\frac{n_1 + d}{n_2-2}\right)^2}\right)^{-1}
\qedhere
\end{align*}

\end{proof}

We will use this lemma when $n_1$ is much smaller than $n_2$, the value $\al$ small, and the distance $d$ grows linearly with $n_2$. In this case we have the following asymptotic version of the stochastic domination.

\begin{corollary}
\label{monotone-cor}
Let $X\left(un + a_n, n, \frac{n}{\sqrt{1 - u^2}}\right)(t)$ be a sequence of tableau processes for $n \in \nat$, where $u \in (-1, 1)$ and $a_n=o(n)$ is a sequence of integers. Then for any $\ep_1 > 2\ep_2 \in (0,1)$, for all sufficiently large $m$ there exists some $N(m)$ such that
$$
X\left(un + a_n, n, \frac{(1+ \ep_1)n}{\sqrt{1 - u^2}}\right) \sgeq X(0, m, m)
$$
up to time $T$, for all values of $n \geq N(m)$.
Here $T$ is the stopping time when $\ep_2 {m \choose 2}$ lattice points have been added to the process $X(0, m, m)$.
\end{corollary}

Finally, we will also state Lemma \ref{stoch-dom-general} for domination of a tableau process over two independently coupled tableau processes, as this will be necessary for the proof that the tableau limit is mixing. The proof goes through analogously.

\begin{lemma}
\label{stoch-dom-indep}
Let $T(b_1, n_1)$ and $T(c_1, n_1)$ be disjoint sets with $T(b_1, n_1) \cup T(c_1, n_1) \sset T(c_2, n_2)$, and consider three tableau processes 
\[
  X_1(t) = X_1(c_1, n_1, n_1), \;\;  X'_1(t) = X_1(b_1, n_1, n_1)\qquad \text{and} \qquad X_2(t) =
  X_2(c_2, n_2, \theta n_2).
\]

Let $Y$ be the process given by the union of independent copies of $X_1$ and $X_1'$. Let $d = \max (|c_2 - c_1|, |c_2 -b_1|)$. Then if $\al$ and $\theta$ are as in the statement of Lemma \ref{stoch-dom-general} (with the new definition for $d$), we have that $Y \sleq X_2$ up to a stopping time $T$. In this case $T$ is the stopping when either $\lfloor \al {n_1 \choose 2} \rfloor$ lattice points have been added to $X_1$ or $X_1'$.
  \end{lemma}

\begin{section}{Inclusion Functions and Convergence}
\label{S:convergence}
%
%
%
%

We want to show that for a sequence of tableau processes $X_n (t) = X_n (0, n, n)(t)$, that the corresponding inclusion functions converge in the weak topology on the space of probability measures on $[0, \infty] ^ \half$. To do this, we use the monotonicity established by Corollary \ref{monotone-cor}, which will be exploited using the following lemmas.

%
\begin{lemma}
\label{domination-implies-convergence}
Let $G_n$ be a tight sequence of random variables taking values in $[0, \infty)^m$. Suppose that for every $\ep > 0$, there exists a sequence of random variables $G^\ep_n$ such that $\prob(G^\ep_n \neq G_n) \to 0$ as $n \to \infty$ and such that the following holds. For all sufficiently large $M$ there is some $N \in \nat$ such that
$$G_n \sleq (1+ \ep)G^\ep_M\qquad \text{ for all } n \geq N.$$
Then the sequence $G_n$ has a distributional limit $G$.
\end{lemma}

We leave the proof of this lemma for the appendix (Section \ref{S:appendix}), as it is fairly standard but somewhat lengthy.

\begin{lemma}
\label{point-tight}
Let $X_n = X_n(a_n, n, n)$ be a sequence of tableau processes  with $a_n =o(n)$ and let $F_n$ be the corresponding sequence of inclusion functions. Then for any $z \in \half$,
$$\{F_n(z) : n \text{  large enough so that } z \in T(0, n)\}$$ is a tight as a sequence taking values in $[0, \infty)$.
\end{lemma}

 \begin{figure}
    \centering
    \includegraphics[scale=0.4]{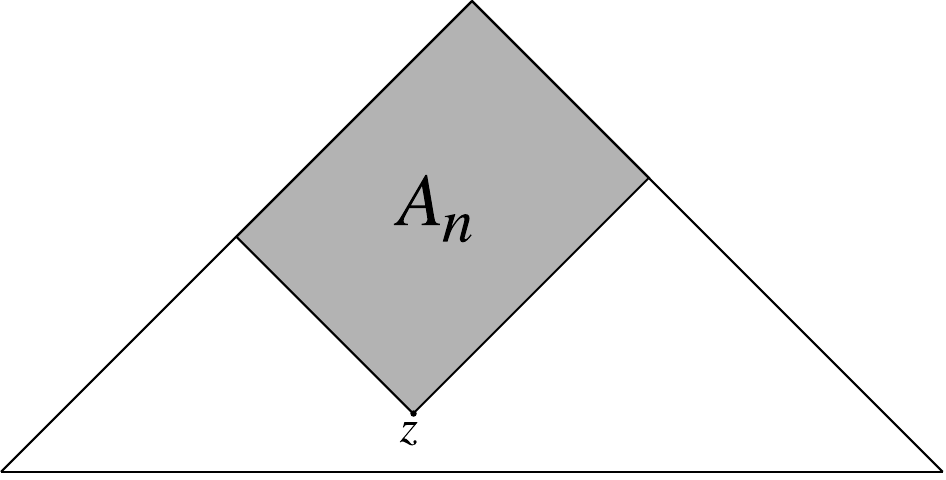} 
    \caption{The set $A_n$ in the proof of Lemma \ref{point-tight}. As the size of the Young diagram goes to infinity, the proportion of $T(0, n)$ taken up by $A_n$ increases to $1/2$, as the point $z$ does not grow with $n$.}
    \label{fig:An}
 \end{figure}

\begin{proof}
Let $z = (x, y)$ and  consider the set $A_n = \{z' \in T(a_n, n) : z' \geq z\}$. Then $A_n$ is a rectangle and as $n\to\infty$ the relative size $|A_n|/|T(a_n,n)|\to 1/2$. Moreover, no square in $A_n$ is added before $z$.

\medskip

Now let $n$ be large enough so that $\card{A_n} > n^2/8,$ and let $m$, $\theta$ be such that
\begin{equation}
  \label{theta-n}
\theta > \frac{n (m-1)}{(n - 1)(m-2)}\left(\frac{1}{4}\sqrt{1- \left(\frac{n + |a_n - a_m|}{m - 2}\right)^2}\right)^{-1}.
 \end{equation}
By Lemma \ref{stoch-dom-general}, $\theta X_m$ dominates $X_n$ until the time when $\frac{3}{4}{n \choose 2}$ squares have been added to $X_n$. By this time at least one square from $A_n$ must have been added to $X_n$, so $z$ must have been added to $X_n$. Therefore  $\theta F_n (z) \sgeq F_m(z)$.

\medskip

As the right hand side of \eqref{theta-n} is bounded uniformly for large $m$ for a fixed value of $n$, there is some $K > 0$ such that  $K F_n(z) \sgeq F_m$ for all large $m$, so $\{F_n(z)\}$ is tight.
 \end{proof}

 Now we can prove Theorem \ref{T:tableau-centre}, which as mentioned previously corresponds precisely with Theorem \ref{T:YT_limit} in the case $u = 0$, and proves all parts of the Theorem in that case except for the mixing property with respect to spatial shift.

\begin{proof}[Proof of Theorem \ref{T:tableau-centre}.]
First assume that $a_n=0$ for all $n$.
Since the product topology on $[0, \infty]^\half$ is compact, $F_n$ has subsequential limits. Suppose that there are two subsequential limits $F^a \neq F^b$. Then for some finite set $K \sset \half$, the restrictions $F^a|_K$ and $F^b|_K$ are not equal. Define $T_n^\ep$ to be the stopping time when $\ep{n \choose 2}$ lattice points have been added to $X_n$. $T_n^\ep \cvgd \infty$ as $n \to \infty$, so for any $z$,  $\prob(F_n(z) \leq T_n^\ep) \to 1$ as $n \to \infty$ since
$$
\{F_n(z) : n \text{  large enough so that } z \in T(0, n)\}
$$
is tight by Lemma \ref{point-tight}. Defining $F^\ep_n$ by  to be $F^\ep_n(z) = F(z)$ for $F(z) \leq T_n^\ep$ and $F^\ep_n(z) = \infty$ otherwise, 
$\prob(F^\ep_n|_K \neq F_n|_K) \to 0$ as $n \to \infty$.

\medskip

Now by Corollary \ref{monotone-cor}, for large enough $m$ there exists $N(m)$ such that $X(0, n, (1 + 3\ep)n) \sgeq X(0, m, m)$ up to time $T_m^\ep$, for all $n \geq N.$ This implies that $(1 + 3 \ep) F^\ep_m(z) \sgeq F_n(z)$. Under these conditions we can appeal to Lemma \ref{domination-implies-convergence}, which gives that $F_n|_K$ does indeed have a distributional limit, contradicting that $F_1|_K \neq F_2|_K$ . Thus $F_n$ itself has some distributional limit $F$. Note that $F \in [0, \infty)^\half$ almost surely since each $\{F_n(z)\}$ is a tight sequence on $[0, \infty)$.

\medskip

The same proof works in the case when $X_n$ is centred at $a_n$ for a sequence $a_n = o(n)$, since all the domination lemmas can be used in exactly the same way. Moreover, translation invariance follows by comparing the sequences $X_n (a_n, n, n)$ and $X_n (a_n + 2, n, n)$ since the difference between the center points, $d_n = 2 = o(n)$.
\end{proof}

\end{section}

\begin{section}{Bounding Rates of Adding Lattice Points}
\label{S:rate}
The goal of this section and the next one is to establish regularity properties of the limit $F$ of random Young tableaux in order to apply the Edelman-Greene bijection. In order to do this we will show that at every time $t$, the points in the limit tableau that are added before time $t$ form a set of disjoint downward closed subsets of $\half$, and that the limit $F$ is still an order-preserving injection. The key to both of these proofs is the following proposition about bounding the rates of adding points in the finite tableau processes.

Throughout this section we let $X_n$ be the tableau process $X_n(0, n, n)$.

\begin{prop}
\label{rate-cor-2} There exist constants $K_1$ and $K_2$ such that for any $z \in \half$  and for any $t$,
$$\expt \left[ \sup_{s \leq t}  v\big(z, X_n(s)\big) \right] \leq K_1t+ K_2,$$
for all large enough $n$ (how large we need to take $n$ depends on the square $z$).
\end{prop}

\bigskip

\noindent {\bf The cylindrical tableau process.} \qquad
To prove this proposition we introduce cylindrical Young diagrams and the cylindrical tableau process. Define $\scrC(n)$, the {\bf discrete cylinder of size $n$}, to be the set of equivalence classes of points $(x, y)$ in $\{ (x, y) \in \half :  1 \leq y \leq n - 1\}$ where $(x, y) \sim (x', y')$ if $y = y'$ and $x \equiv x' \pmod{2(n-1)}$. This cylinder has the following partial order inherited from the partial order on $\half$. For $(x, y), (x', y') \in \scrC(n)$, $(x, y) \leq (x', y')$ if $(x', y') \sim (x'', y'')$ for some $(x'', y'') \in \half$ with $(x', y) \leq (x'', y'')$. 

Thus we have a notion of downward closed sets in $\scrC(n)$, and notions of corners, hooks, and reverse hooks in $\scrC(n) \setminus A$ for any downward closed set $A \sset \scrC(n)$ by thinking of $\scrC(n)$ as a cylindrical Young diagram. As in a usual Young diagram, for any corner $z \in \scrC(n) \setminus A$ we can define the ``hook probability" for $z$ by

\begin{equation*}
\hookp(\scrC(n)\setminus A, z) = \frac{1}{\card{\scrC(n) \setminus A} }\prod_{y \in R_z} \left(1 + \frac{1}{h_y -1} \right).
\end{equation*}

 Now we define the {\bf cylindrical tableau process} $C(t) = C(n, r)(t)$ on $\scrC(n)$ with rate $r$ as the continuous time Markov process $C(t)$ where a square $z$ is added to configuration $A$ at rate
$$v_C(z,A)=r \hookp (\scrC(n)\setminus A,z).$$ Note that the hook probabilities in  cylindrical tableaux do not sum to 1 as they do with staircase tableaux. This is not an issue as we are only using the hook probabilities to define rates, not as actual probabilities.

\medskip

The symmetry in the cylindrical process makes it easier to bound the expectation of the rate $v_C(z, C(t))$. We can then use that the staircase tableau process can be coupled with an appropriately sped up cylindrical process in a way that allows rates in the staircase process to be controlled by the rates in the cylindrical process. This will prove Proposition \ref{rate-cor-2}.

\bigskip

\noindent {\bf The modified rate.} \qquad Instead of working with $v(z, A)$, we will replace it with a monotone increasing function $w(z,A)$ called the modified rate.  The {\bf modified rate} $w_C(z, A)$ is the rate of adding $z$ to the configuration $A$ with the cone
$
S_z=\{z' : z' \geq z\}
$
above $z$ removed. More precisely
$$
w_C(z,A)=v_C(z,A\setminus S_z).
$$

\medskip

By the definition of $v$, the modified rate satisfies
$$w_C(z, A) = \frac{r}{\card{\scrC(n) \setminus A\cup S_z }} \prod_{y \in R_z} \left(1 + \frac{1}{f^z_y -1}\right).$$
Here $f_y^z$ is the hook length of $y$ in the residual tableau corresponding to the state $A\setminus S_z$.
%
%
%
We also define $w_X(z, A)$ for a staircase tableau process $X$ in the analogous way. 

\medskip

Since $v_C$ is monotone in $A$ as long as $z_C$ has not been added, we get that $w_C$ is monotone in $A$ (even if $z$ has been added). Therefore to prove Proposition  \ref{rate-cor-2} it suffices to prove the following.

\begin{prop} 
\label{mod-rate}
For all large enough $n$, we have
\begin{equation}  \label{mod-rate-bound}
\expt \left[ \sup_{s \leq t}  w_X(z, X_n(s)) \right] = \expt w_X(z, X_n(t)) \leq K_1t+ K_2
\end{equation}

\end{prop}

We first need a lemma bounding the products in the hook probability formula.

\begin{lemma}
\label{maxheight}
Let $A$ be a downward closed subset of $\mathcal C(n)$, and let $\beta$ be the maximal second coordinate of squares in $A$. Then we have
\begin{align*}
\prod_{y \in R_z} \left(1 + \frac{1}{f^z_y(A) -1}\right)
< 2 n (\beta + 1).
\end{align*}
\end{lemma}
\begin{proof}
If we order squares in the reverse hook of $z$ by their second coordinate ($s+1$ below), we get upper bounds on the individual factors. This gives an overall upper bound
\begin{align*}
\prod_{s = 1}^{n-1} \left(1 + \frac{1}{s + (s - \beta)^+}\right)^2
= (\beta + 1)^2 \prod_{s = (\beta + 3)/2}^{n - 1 - \beta} \left(\frac{2s}{2s -1} \right)^2< 2 n (\beta + 1).
\end{align*}
The last inequality is from Lemma \ref{hook-prob-inequality}.
\end{proof}

\begin{remark} The same bound holds in the staircase tableau case.
\end{remark}

Next, we bound $w(z, A)$ for $z$ at the bottom of the cylinder.

\begin{prop}
\label{bounding-rate-in-exp}
Let $\mathcal B$ denote the bottom row of $\mathcal C(n)$. Then we have that 
$$\sum_{z\in \mathcal B} w(z, A) \leq 48\left(|A| + n\right),$$
in the rate $n$ cylindrical tableau process.
\end{prop}

We have only included the explicit constant 48 in the above proposition to streamline the proof. It is far from optimal for large $n$.
\begin{proof}
For $z\in \mathcal B$, define
$$
D^z=\prod_{y \in R_z} \left(1 + \frac{1}{f^z_y(A) -1}\right).
$$
It suffices to show that
\begin{equation}
\label{rate-bound-2}
\sum_{z\in\mathcal B} D^z \leq 12 (n |A|+ n^2),
\end{equation}
since $\card{\scrC(n) \setminus A\cup S_z} \geq \card{S_z} \ge {n \choose 2}$. To establish this bound, we will build the set $A$ in $|A|$ steps by starting with $A_0=\emptyset$ and repeatedly adding a single square $(\alpha_i,\beta_i)$ to $A_{i-1}$ to get $A_i$. We do this in a way so that $A_i$ stays downward closed and $\beta_i$ are non-decreasing.

Define the quantities $D^z_i$ for $A_i$ analogously to $D^z$.
By simple algebra,
$$
\sum_{z\in\mathcal B} D^z \le n\max_{z\in \mathcal B}D_0^z+\sum_{i=1}^{|A|}\left[\max_{z\in \mathcal B}{D^z_{i-1}}\right]\sum_{z\in \mathcal B}\left|\frac{D^z_{i}}{D^z_{i-1}} - 1\right|.
$$
By defining $\beta_0=0$, we have that $\beta_i$ is the maximal $y$-coordinate of a square in $A_i$.
The first term on the right is bounded above by $2n^2$ by Lemma \ref{maxheight}. By the same lemma, 
$$
\max_{z\in \mathcal B}{D^z_{i-1}}\le 2n(\beta_{i-1}+1)\le 2n(\beta_i+1) \leq 4n \beta_i,
$$
since $\beta_i \geq 1$ for $i \geq 1$. So it suffices to show that for any $ i \ge 1$, we have
\begin{equation}\label{ratios}\sum_{z\in \mathcal B}\left|\frac{D_{i}^z}{D_{i-1}^z}-1 \right|\le \frac{3}{\beta_i}.
\end{equation}
To do this, recall that $A_{i}=A_{i-1}\cup\{(\alpha_i,\beta_i)\}$. Note that if for $z\in \mathcal B$ we have $z\le (\alpha_i,\beta_i)$ then $D_{i}^z/D_{i-1}^{z}=1$. Let $\mathcal B'=\mathcal B\setminus \{z:z \le (\alpha_i,\beta_i)\}$. Then $|\mathcal B'|=n-1-\beta_i$.

For any $z\in \mathcal B'$ the reverse hooks $R_z$ and $R_{(\alpha_i,\beta_i)}$ intersect at exactly two points, one on the right leg of $R_z$ and one on the left. Call the $y$-coordinates of these points $s_z$ and $s'_z$, respectively. As we move $z$, these intersection points exhaust the set $R_{(\alpha_i,\beta_i)}$. More precisely, $s$ and $s'$ are both bijections from $\mathcal B'$ to  $\{\beta_i+1,\ldots, n-1\}$. For $z\in \mathcal B'$ we have
$$
D_{i}^z/D_{i-1}^{z}=Q(s_z)Q(s'_z),
$$
where
$$
Q(s)=
\frac{1 + \frac{1}{2s - \beta_i  - 2}}{1 + \frac{1}{2s - \beta_i - 1}} =  \frac{(2s - \beta_i -1)^2}{(2s - \beta_i- 2)(2s - \beta_i)}.
$$
Since $s$ and $s'$ are bijections, Cauchy-Schwarz gives
$$
\sum_{z\in \mathcal B'} Q(s_z)Q(s'_z) \le \sum_{s=\beta_i+1}^{n-1} Q(s)^2.
$$
By simple algebra $Q(s)\ge 1$ and
$$
Q(s)^2-1\le \frac{3}{(2s-\beta_i - 1)^2}.
$$
So the left hand side of \eqref{ratios} is bounded above by
\[
\sum_{s=\beta_i+1}^{n-1}\frac{3}{(2s-\beta_i - 1)^2}< \frac{3}{\beta_i}. \qedhere
\]
\end{proof}

Now we can embed the staircase tableau of size $n$ into the cylinder $\scrC(n)$ by identifying the subset $T(0, n)$ with its equivalence class in $\scrC(n)$. Thus we can talk about stochastic domination of a cylindrical tableau process over a staircase tableau process, and we can talk about domination of modified rates.

\begin{lemma}
\label{cyl-dom}
Let $T$ be the time at which $n^2/4$ particles have been added to the tableau process $X_n(t)$. Let $C_n(t)$ be a cylinder process on $\scrC (n)$ with rate $8n$. Then there exists a coupling so that for $n \geq 3$,
$$
X_n(t) \le C_n(t) \qquad \text{ for all }t\le T.
$$
Moreover, for any $z$ in the bottom row of $T(0, n)$, $w_C(z, C_n(t)) \geq w_X(z, X_n(t))$ for all $t \leq T$ in this coupling.
\end{lemma}

\begin{proof}
To prove the existence of a coupling, it suffices to show that for any $A$, and any lattice point $z$ that is both a corner of $\scrC(n)\setminus A$ and $T(0, n) \setminus A$, that
$v_{X_n}(z, A) \leq v_{C_n}(z, A).$
From here we can appeal to Lemmas \ref{stoch-dom-criterion} and \ref{stoch-dom-1}, which can be proven in the exact same way if one of the processes is a cylinder process.

\medskip

Reverse hooks in $\scrC(n)$ are larger than reverse hooks in $T(0, n)$, and for $y \in \scrC(n)\cap T(0, n)$, we have $h_y^X = h_y^C$, so
\begin{equation}
\label{hook-cyl}
 \prod_{y \in R^z_X} \left(1 + \frac{h^X_y}{h^X_y -1} \right) \leq \prod_{y \in R^z_C} \left(1 + \frac{h^C_y}{h^C_y -1} \right).
\end{equation}
 Also,

\begin{equation}
\label{size-cyl}
\frac{\card{\scrC(n)\setminus A}}{\card{T(0, n) \setminus A}} \leq 8
\end{equation}
for all $n \geq 3$. Combining the inequalities \eqref{hook-cyl} and \eqref{size-cyl} proves the lemma. The relation among modified rates follows in the same way.
\end{proof}

Now we can prove Proposition \ref{mod-rate} for $z$ in the bottom row of $T(0, n)$.

\begin{proof}
Let $T$ be the stopping time when $(2t+1) n$ squares have been added to the tableau process $X_n$. Since the times of adding squares are the points of a rate $n$ Poisson process, it is easy to check that
$$
\prob (T<t) \leq e^{-L n}
$$ for some universal constant $L$.

Observe the naive bound that $w_X(z, X_n(t)) \leq n$ for all $n$. We can now use Lemma \ref{cyl-dom} together with the monotonicity of modified rates to get:
\begin{align*}
\expt \left[ \sup_{s \leq t}  v_X(z, X(s)) \right] &\leq \expt w_X(z, X_n(t)) \\
&\leq  \expt \left[ w_C(z, C_n(t)) \, \mathbf 1_{t < T}\right] + ne^{-L n} \\
&\leq \expt w_C(z, C_n(T))  + ne^{-L n}.
\end{align*}
Finally, using Proposition \ref{bounding-rate-in-exp} and the rotational symmetry of the cylinder process, we get that
$$\expt w_C(z, C_n(T))  + ne^{-L n} \leq 48((2t+1) + 1) + ne^{-L n} \leq K_1t + K_2,$$
completing the proof.
\end{proof}

Finally, we show that for any fixed $z' \geq z \in \HH$, that for large enough $n$, the modified rate for adding $z'$ to $X_n$ is always bounded by twice the modified rate for adding $z$. 
This extends Proposition \ref{mod-rate} to encompass all $z \in \HH$, and therefore completes the proof of Proposition \ref{rate-cor-2}.

\begin{lemma} Let $z' \geq z = (z_1, z_2)$ and for a downward closed subset $A \sset T(0, n)$ let $w(z, A)$ and $w(z', A)$  be the modified rates in $X_n$. Then
\begin{equation}
\label{asym-larg}
\lim_{n \to \infty} \left(\sup_{A \sset T(0, n)} \frac{w(z', A)}{w(z, A)} \right) = 1.
\end{equation}
Specifically, for all large enough $n$, we have that $w(z', X(t)) \le 2w(z, X(t))$ for all $t$.
\end{lemma}

\begin{proof}  We only prove this in the case $z' = (z_1 + 1, z_2 +1)$, as the general case follows by symmetry and induction. Observe first that the supremum on the right hand side of \eqref{asym-larg} is at least 1 for every $n$, since $w(z', T(0, n)) = w(z, T(0, n))$ for all $n$. Also, it is easy to see that
$$\frac{\card{ T(0, n) \setminus A \cup S_z}}{\card{ T(0, n) \setminus A \cup S_{z'}}} \to 1$$
as $n \to \infty$, since $\card{S_z}/n^2 \to 1/4$ as $n \to \infty$, but $\card{S_z \symdif S_{z'}}/n \to 1/2$. Therefore to complete the proof it suffices to show that for any configuration $A \sset T(0, n)$, that

\begin{equation}
\label{hookprod}
\prod_{y \in R_{z'}} \left(1 + \frac{1}{f^{z'}_y -1}\right) \le \prod_{y \in R_z} \left(1 + \frac{1}{f^z_y -1} \right).
 \end{equation}

To prove this, let $y' = (y_1 + 1, y_2 + 1) \in R_{z'}$. It is clear that $y = (y_1, y_2)$ must be in $R_z$. Moreover, if $(x_1, x_2) \in H_y$ in the configuration $T(0, n) \setminus A \cup S_z$, then $(x_1 + 1, x_2 + 1) \in H_{y'}$ in the configuration $T(0, n) \setminus A \cup S_{z'}$. This gives an injective mapping of $R_{z'}$ into $R_z$ that does not decrease hook length, proving \eqref{hookprod}.
\end{proof}

\end{section}

\begin{section}{Regularity and Mixing of the Limit $F$}
\label{S:regularity}
In Theorem \ref{T:tableau-centre} we showed that the inclusion functions of random staircase Young tableaux have a limit $F$. In this section we establish regularity properties and mixing of $F$ using the results of Section \ref{S:rate}. 

\begin{prop}
\label{inject}
$F$ is almost surely injective.
\end{prop}

\begin{proof}
Suppose not. Since there are only countably many pairs of points in $\half$, then there exists a pair $(z_1, z_2) \in \half^2$ with $\prob(F(z_1) = F(z_2)) = \delta > 0$. Then for any $\ep > 0$, there is some $N$ such that $\prob (|F_n(z_2) - F_n(z_1)| < \ep) \geq \frac{\delta}{2}$ for all $n \geq N$. Without loss of generality, we can remove the absolute values at the expense of a factor of $1/2$ to get
\begin{equation}
\label{close-rates}
\prob (0\le F_n(z_2) - F_n(z_1) < \ep) \geq \frac{\delta}{4}.
\end{equation}

Let $T$ be the stopping time when $z_1$ is added to the process $X_n$. The probability of adding $z_2$ in the interval $[T, T + \ep]$ is bounded by the integral of the rate in that interval. This gives that
\begin{align*}
\prob (0 \le F_n(z_2) - F_n(z_1) < \ep) &\leq \sup_{s \in [T, T + \ep]} v(z_2, X_n(t)) \\
&\leq Kt \ep + \prob(\sup_{s \leq t} v(z_2,  X_n(s)) \geq Kt) + \prob(T > t - \ep).
\end{align*}
By Proposition \ref{rate-cor-2} we can choose $K$ and $t$ large enough and independently of $\ep$ to make the last two terms on the right hand side arbitrarily small for all large enough $n$. Taking $\ep$ close to 0 then contradicts \eqref{close-rates}.
\end{proof}

\begin{corollary} 
\label{atom}
For each $z$, the distribution of $F(z)$ has no atoms. 
\end{corollary}

\begin{proof}
The proof that $F$ has no atoms is the same as the proof that $F$ is almost surely injective, except instead of conducting the analysis at a stopping time $T$ when the square $z_1$ is added, we conduct it at a (deterministic) time $t$. 
\end{proof}

\begin{subsection}{The limit $F$ is mixing}
Recall that a measure $\mu$ is $k$-mixing with respect to a measure-preserving transformation $\tau$ if for any measurable sets $A_1, \ddd, A_k$,
\begin{equation*}
\lim_{m_1, \ldots, m_k \to \infty} \mu(A_1 \cap \tau^{-m_1}A_2 \cap \ldots \cap \tau^{-m_1 - m_2 - \ddd - m_k} A_k) = \prod_{i=1}^k \mu(A_i).
\end{equation*}
 Note that this proposition completes the proof of Theorem \ref{T:YT_limit} in the case $u =0$.

\begin{prop}
\label{asym-ind}
The limit $F$ is mixing of all orders with respect to the spatial shift $\tau$.
\end{prop}

We first present an outline of the proof that $F$ is 2-mixing. Fix $m$, and consider two sets $A^r$ and $B^r$ of the form
$$A^r =  \prod_{i \in T(0, m)} [0, a_i] \;\; \mathand  \;\; B^r = \prod_{i \in T(0, m)} [0, b_i],$$
and let 
$$ A = A^r \times \prod_{i \notin T(0, m)} [0, \infty) \;\;\mathand\;\; B = B^r \times \prod_{i \notin T(0, m)} [0, \infty).$$
By Dynkin's  $\pi-\la$ Theorem, it suffices to show that
\begin{equation}
\label{necessary-for-ergod}
\prob(F  \in A \cap \tau^{-K} B) \to \prob(F \in A)\prob(F \in B) \mathas K \to \infty,
\end{equation}
for any such $A$ and $B$. To show this, we will approximate the value of $F$ on $A \cap \tau^{-K} B$ in two different ways. Figure \ref{fig:2approx} illustrates the two approximations used. For the first approximation,  take two disjoint tableaux $T(0, \lfloor K/2 \rfloor)$ and $T(-K, \lfloor K/2 \rfloor)$ and run independent, rate-$\lfloor K/2 \rfloor $ tableau processes $Y_1$ and $Y_2$ on each of these tableaux. Let $G_{K, 1}$ and $G_{K, 2}$ be the inclusion functions for $Y_1$ and $Y_2$. For $K \gg m$, convergence of $G_{K, 1}$ and $G_{K, 2}$ to $F$ implies that $\prob(G_{K, 1} \in A)$ is very close to $\prob(F \in A)$, and similarly for $G_{K, 2}$ and $\tau^{-K}B$.

\medskip

For the second approximation, take $n \gg K$, and let $X_n$ be the rate-$n$ tableau process on $T(0, n)$ with inclusion function $F_n$. Since $n \gg K$, the convergence of $F_n$ to $F$ implies that 
$\prob(F_n \in A \cap \tau^{-K} B)$ is close to $\prob(F \in A \cap \tau^{-K} B)$, and that $\prob(F_n \in A)$ and $\prob(F_n \in \tau^{-K} B)$ are close to $\prob(F \in A)$ and $\prob(F \in B)$ respectively.

\medskip

Finally, we can use the domination Lemma \ref{stoch-dom-indep} to show that a small speed-up of $X_n$ dominates the union of the independent processes $Y_1$ and $Y_2$ up to a large stopping time. This in turn implies that up to a small error, 
$$\prob(F_n \in A \cap \tau^{-K} B) < \prob(G_{K, 1} \in \beta A)\prob(G_{K, 2} \in \beta \tau^{-K} B),$$ where $\beta$ is the value of the speed-up. Combining this with our previous relationships between probabilities implies that $\prob(G_{K, 1} \in A)\prob(G_{K, 2} \in \tau^{-K} B)$ must be very close to $\prob(F_n \in A \cap \tau^{-K} B)$. Passing to the limit in $n$ and then $K$ then proves that $F$ is 2-mixing, noting that 
$$\prob(G_{K, 2} \in \tau^{-K} B) \to \prob(F \in B) \mathas n \to \infty$$
by spatial stationarity.

\medskip

The general case can be proven using the same method, with the main difference being that in that case, we approximate the limit $F$ with $n$ disjoint independent tableau processes instead of 2. For simplicity, we only prove 2-mixing below.

\begin{figure}
   \centering
   \includegraphics[width=.9\textwidth]{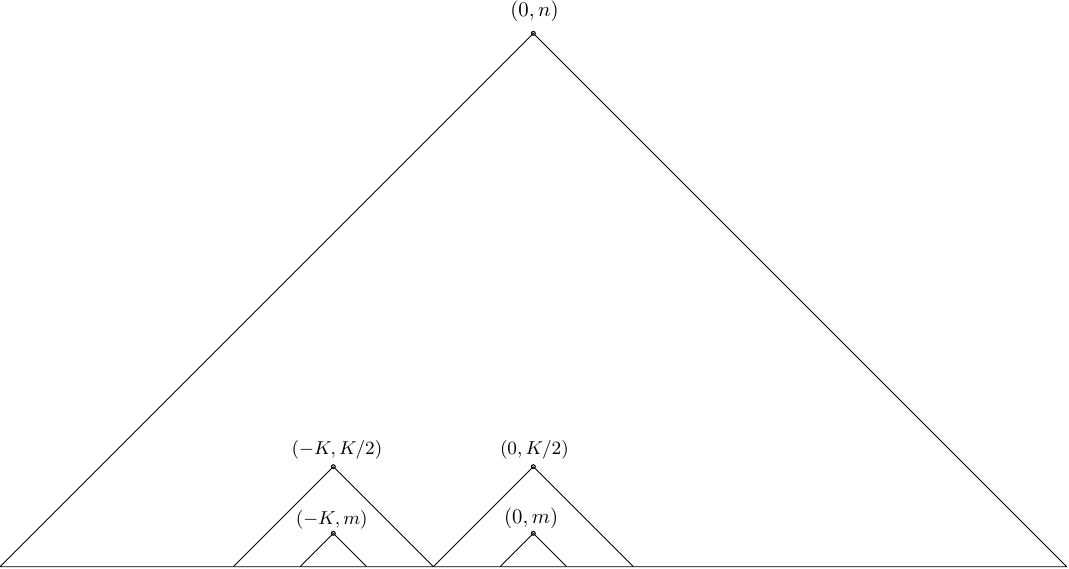}
   \caption{The two approximating processes used in the proof of Proposition \ref{asym-ind}. The first approximation pairs two disjoint processes on $T(0, \lfloor K/2 \rfloor)$ and $T(-K, \lfloor K/2 \rfloor)$ for $K \gg m$ and the second approximation takes a tableau process on $T(0, n)$ for $n \gg K$.}
   \label{fig:2approx}
\end{figure}

\begin{proof} The proof exactly follows the outline of what is stated above, but with precise bookkeeping regarding the error terms.  

%
%

With notation as in the outline, first note that it suffices to show that for large enough $n$,
\begin{equation}
\label{H-G}
|\prob(F_n \in A \cap \tau^{-K} B) - \prob(G_{K, 1} \in A)\prob(G_{K, 2} \in \tau^{-K} B)| < \ep_K
\end{equation}
where $\ep_K \to 0$ as $K \to \infty$. To see that \eqref{H-G} implies \eqref{necessary-for-ergod}, let 
$$
|\prob(F \in A)\prob(F \in B) - \prob(G_{K, 1} \in A)\prob(G_{K, 2} \in \tau^{-K}B)| = \delta_K.
$$
 Taking $n \to \infty$ in \eqref{H-G} and replacing $G_{K, 1}$ and $G_{K, 2}$ with $F$, we get that
\begin{equation*}
|\prob (F \in A \cap \tau^{-K} B) - \prob(F \in A)\prob(F \in B)| \leq \ep_K + \delta_K.
\end{equation*}
We can pass to the limit in $F_n$ since $A \cap B$ is a set of continuity of $H$ by Corollary \ref{atom}. Moreover, $\delta_K \to 0$ as $K \to \infty$ since $A$ and $B$ are sets of continuity of $F$ by the same corollary and using the spatial stationarity of $F$.

\medskip

Now let $\gamma > 0$, define $$\beta =  \frac{K + 1}{(1- \gamma)(K - 4)},$$ and let
\begin{equation*}
\label{eqn-3L}
\alpha_{K, \ga} = \max \left\{\prob\left(F(i) \in \left[c, \beta c \right] : c \in \{a_i, b_i\}\right)\right\}.
\end{equation*}
We have chosen $\beta$ in a way so that if $n$ is large enough, then the tableau process $X^\beta_n$ on the tableau $T(0, n)$ with speed $\beta n$ stochastically dominates the independent coupling of the tableau processes $Y_1$ and $Y_2$ up to time $T_\ga$. Here $T_\ga$ is the time when either $\ga {\lfloor K/2 \rfloor \choose 2}$ squares have been added to $Y_1$ or $\ga {\lfloor K/2 \rfloor \choose 2}$ squares have been added to $Y_2$. This can be seen by comparing with the condition in Lemma \ref{stoch-dom-indep}. 

\medskip

Therefore letting $M = \max \{a_i, b_i\}$, we have
$$\prob(G_{K, 1}\in A)\prob(G_{K, 2} \in \tau^{-K} B) < \prob\left(\frac{F_n}{\beta} \in A \cap \tau^{-K}B\right) + \prob(T_\gamma < M).$$
Moreover, we have that for all large enough $n$, 
\begin{align*}
\prob\left(\frac{F_n}{\beta} \in A \cap \tau^{-K} B \right) &= \prob\left(F_n \in\beta A \cap \tau^{-K} \beta B \right)\\
&< \prob(F_n \in A \cap \tau^{-K} B) + 2m(m-1)\al_{K, \ga}.
\end{align*}
Here $\beta A = \{ \beta x : x \in A \}$, and similarly for $B$. For the above inequality to hold, $n$ just needs to be large enough so that 
$$
\max \big\{\prob\left(F_n(i) \in \left[c, \beta c \right] : c \in \{a_i, b_i\}\right)\big\} < 2 \al_{K, \ga}.
$$ 
Combining the above two inequalities, we get that 
\begin{equation}
\label{bd-one-way}
\prob(G_{K, 1}\in A)\prob(G_{K, 2} \in \tau^{-K} B) < \prob(F_n \in A \cap \tau^{-K} B) + 2m(m-1)\al_{K, \ga} + \prob(T_\gamma < M).
\end{equation}
We can similarly get that
\begin{equation}
\label{bd-two-way}
\prob(G_{K, 1}\in A^c)\prob(G_{K, 2} \in \tau^{-K} B^c) > \prob(F_n \in A^c \cap \tau^{-K} B^c) - 2m(m-1)\al_{K, \ga} - \prob(T_\gamma < M).
\end{equation}
Finally, let 
$$
\sig_K = \max\{|\prob(G_{K, 1}\in A) - \prob(F \in A)|,  |\prob(G_{K, 2}\in \tau^{-K}B) - \prob(F \in \tau^{-K} B)|\}.
$$ 
For large enough $n$, we have that 
\begin{equation}
\label{close-marg}
|\prob(F_n \in A) - \prob(G_{K, 1} \in A)| < 2\sig_K \mathand |\prob(F_n \in \tau^{-K}B) - \prob(G_{K, 1} \in \tau^{-K}B)| 
\end{equation}
since $A$ and $B$ are sets of continuity of $F$. This similarly holds for $A^c$ and $B^c$. Therefore
 \begin{align*}
\big|\big(\prob(G_{K, 1} \in A)\prob(G_{K, 2} \in \tau^{-K} B)& - \prob(G_{K, 1} \in A^c)\prob(G_{K, 2} \in \tau^{-K} B^c)\big)\\
 - &\big(\prob(F_n \in A \cap \tau^{-K} B) - \prob(F_n \in A^c \cap \tau^{-K} B^c)\big) \big| < 4 \sig_K.
 \end{align*}
 Combining this bound with \eqref{bd-one-way} and \eqref{bd-two-way} gives that for all large enough $n$,
 \begin{align}
 \label{almost}
 \begin{split}
\big |(\prob(G_{K, 1} \in A)\prob(G_{K, 2} \in \tau^{-K} B) &- \prob(F_n \in A \cap \tau^{-K} B)\big|\\
& < 4m(m-1)\al_{K, \ga} + 2\prob(T_\ga < M) + 4 \sig_K.
\end{split}
 \end{align}
 
Now note that for any fixed value of $\ep > 0$, as $K \to \infty$, we can choose a sequence $\ga_K \to 0$ such that $\prob(T_{\ga_K} > M) < \ep$. With this sequence of $\ga_K$s, $\al_{K, \ga_K} \to 0$ as $K \to \infty$. 

\medskip
Noting also that $\sig_K \to 0$ as $K \to \infty$, this shows that the left hand side of \eqref{almost} tends to 0 as $n \to \infty$. 
\end{proof}

%
%
%
%
%
%
%

\end{subsection}

\begin{subsection}{Inclusion times for squares in the bottom row}
We can also use the rate bound to get a lower bound on the probability that it takes a long time to add any given square in the bottom row. Note that by spatial stationarity of the limit $F$, it suffices to prove this for the square $z_0 = (1, 1)$. We can then combine this with the mixing property of $F$ to show that at any time infinitely many squares have not been added.

\medskip

 The idea here is to modify the process $X_n$ to create a new process $Y_n$. $Y_n$ will be  $X_n$, but with the hook probabilities modified so that $Y_n$ never adds $z_0$. We will then show that $X_n$ and $Y_n$ can always be coupled so that at any time $t$ they are equal with positive probability $P$ independent of $n$.
\bigskip

\noindent {\bf The construction of $Y_n$}. $Y_n$ is a Markov process with the same state space as the tableau process $X_n$, namely:
$$\{ A \sset T(0, n) : A \text{   is downward closed} \}.$$
If $A = Y_n(t)$, and  $z$ is corner of $T(0, n) \setminus A$ with $z \neq z_0$, then we add the point $z$ to $Y_n$ with rate

$$n \frac{\hookp(T(0, n) \setminus A, z)}{1 - \hookp(T(0, n) \setminus A, z_0)}.$$
In words, the rates in $Y_n$ for squares that can be added are given by the rates in $X_n$ times
\begin{equation}
  \label{XnYnrates}
  \frac{1}{1 - \hookp(T(0, n) \setminus A, z_0)}.
\end{equation}
Note that this only makes sense as long as there are squares other than $z_0$ that can be added to $Y_n$. Once $z_0$ is the only square left that can be added, we can define $Y_n$ so that nothing happens past that point. We first show that $Y_n$ is dominated by a sped-up version of $X_n$. Note that the total rate of jumps from any non-terminal state in $Y_n$ is exactly $n$.

\begin{lemma} \label{X2domY}
Let $M < \frac{n(n-1)}{128}$, and let $T_{M}$ be the stopping time when $M$ squares have been added to $Y_n$. Then letting $X^2_n = X(0, n, 2n)$ be a tableau process on $T(0, n)$ with speed $2n$, we have that $Y_n \sleq X^2_n$ up to time $T_M$.
 \end{lemma}

\begin{proof} Suppose that $A$ is some configuration with fewer than $\frac{n(n-1)}{128}$ points added. The maximum height of $A$ is bounded by $\frac{n}{8} - 1$, since any square of height $\frac{n}{8} - 1$ lies above a triangle with $\frac{n(n-1/8)}{128}$ squares. By the remark following Lemma \ref{maxheight} this implies a  bound on the hook probabilities, namely
$$
\hookp(T(0, n) \setminus A, z_0) \leq \frac{1}{{n \choose 2} - M} 2 n \frac{n}{8} < \frac{1}{2}
$$
for $n \geq 3$. Then by \eqref{XnYnrates} we have domination of the rates of $Y_n$ by those in $X^2_n$. Lemmas \ref{stoch-dom-criterion} and \ref{stoch-dom-1} (more precisely, the proofs of those lemmas,) then imply stochastic domination.
\end{proof}
Now we couple $X_n$ and $Y_n$ to bound the probability of adding $z_0$.
\begin{prop}
\label{prob-not-added}
There exist constants $K$ and $L$ such that for any $t>0$

$$\prob\{F(z_0) > t \} \geq e^{-Kt -L t^2}.$$

\end{prop}

\begin{proof} Couple $X_n$ and $Y_n$ so that they add squares at the same times (we can do this since the total rate of exiting non-absorbing states in $X_n$ and $Y_n$ is $n$), and add the same squares until the time when $X_n$ adds square 0. Now let $M \in \nat$, and for $m \in \nat$ let $T_m$ be the stopping time when the $m$th square is added to $X_n$. Let $\scrM$ be the set of maximal sequences $\{A_0 = \emptyset \sset A_1 \ddd \sset A_M\}$ of downward closed subsets of $T(0,n)$  such that $z_0 \notin A_M$.  Then we have
$$
\prob(X_n(t) = Y_n(t)) \geq \prob(T_M \geq t)\sum_{\{A_m\} \in \scrM} \prob(X(T_m) = A_m \mathforall m \leq M) $$
Using the transition probabilities for $Y_n$ the sum above can be written as
$$\sum_{\{A_m\} \in \scrM} \prob(Y(T_m) = A_m \mathforall m \leq M) \prod_{m=1}^{M} (1 - \hookp( T(0, n) \smin A_{m-1}, z_0   )).
$$
We may write this as an expectation
$$\expt \prod_{m=1}^{M} \big(1 - \hookp( T(0, n) \smin Y_n(T_{m-1}) ,z_0   )\big)\ge \expt \big(1 - \hookp( T(0, n) \smin Y_n(T_{M}), z_0   )\big)^M.
$$
The
 inequality follows since the probabilities are monotone. 
By Jensen's inequality we get the lower bound
$$
\big(1 - \expt \hookp( T(0, n) \smin Y_n(T_{M}), z_0   )\big)^M
$$
We use Lemma \ref{X2domY} to bound the expectation above. Assume 
$M \leq \frac{n(n-1)}{128}$, then $X^2_n(T_M)$ stochastically dominates $Y_n(T_M)$, that is in some coupling $X^2_n(T_M)\ge Y_n(T_M)$, and since $z_0\notin Y_n(T_M)$, we have $X^2_n(T_M)\smin S_{z_0}\ge Y_n(T_M)$, where $S_{z_0}$ is the set of squares that are greater than $z_0$ in the partial order. By monotonicity of the rates we have
$$
\expt \hookp( T(0, n) \smin Y_n(T_{M}), z_0   ) \le \expt \hookp( T(0, n) \smin (X^2_n(T_{M})\setminus S_{z_0}), z_0   ).
$$
We can bound the rates in $X^2_n \smin S_{z_0}$ at some fixed time $s=2M/n$ by Proposition \ref{mod-rate} from the previous section. Here note that the rate of adding $z_0$ to $X^2_n \smin S_{z_0}$ is the modified rate of adding $z_0$ to $X_n$.  We get the upper bound
$$
\expt \hookp( T(0, n) \smin (X^2_n(s)\setminus S_{z_0}), z_0   ) + \prob(T_M>s) \le \frac{K_1}{n}+\frac{K_2 M}{n^2} + e^{-K_3n},
$$
where bound on $\prob(T_M > s)$ follows from the tail probabilities of the Poisson distribution. Monotonicity of the rates implies that this is also an upper bound for $\expt \hookp( T(0, n) \smin (X^2_n(T_{M})\setminus S_{z_0}), z_0)$.
Putting everything together and setting $M=2tn$, we get for large enough $n$
\[
\prob(X_n(t) = Y_n(t)) \geq \prob(T_M > t)\left(1 - \frac{K_1+2 K_2 t}{n}-e^{-K_3n}\right) ^{2tn}.
\]
Letting $n\to\infty$ gives that 
$$\lim_{n \to \infty} \prob(X_n(t) = Y_n(t)) \geq e^{-Kt - Lt^2},$$
for some constants $L$ and $K$. Using that $F(z_0)$ has a continuous distribution (Corollary \ref{atom}) then finishes the proof.
\end{proof}

The mixing of $F$ combined with Lemma \ref{prob-not-added} implies that at any time, a bi-infinite sequence of squares has not been added. This is a direct consequence of the fact that mixing implies ergodicity.

\begin{corollary}
\label{cor-spaces}
For any time $t$, there are almost surely infinitely many values of $x > 0$ and infinitely many values of $x < 0$ such that $F(x, 1) > t$.
\end{corollary}

\end{subsection}

\begin{section}{Sorting Networks at the Center}
\label{S:llcentre}
Now we are finally in a position to prove the existence of the local limit of random sorting networks at the center. Let $\scrA$ be the space of swap functions. Define $\scrI$ to be the set of all functions $G: \half \to [0, \infty]$ such that the following two conditions hold. 

\medskip

i) Let $B = \{z \in \half : G(z) \neq \infty\}$. Then $G|_B$ is order-preserving and injective.

ii) For any $t$, we have that  $G(x, 1) > t$ for infinitely many $x> 0$ and $x < 0$.

\medskip

\noindent We will define a map $\oper{EG}: \scrI \to \scrA$ which will generalize the Edelman-Greene bijection. To do this we first define swap functions $\oper{EG}_t(G)$ for every $t > 0$. These swap functions will be $\oper{EG}(G)$ defined up to time $t > 0$. Consider the set of points $$A = \{z : G(z) \leq t\} \sset \half.$$ Since $G(x, 1) > t$ for infinitely many $x> 0$ and $x < 0$ and $G$ is order-preserving on $\half$, $A$ breaks down into infinitely many finite downward closed sets $A_i$ such that each $A_i$ lies in some $T(\ell_i, k_i)$ and the sets $T(\ell_i, k_i)$ are disjoint. We can then define the swap function on each $T(\ell_i, k_i)$ individually up to time $t$ using the regular Edelman-Greene bijection on that diagram, since these swap functions don't interact before time $t$ and $G|_B$ is order-preserving and injective.

\medskip

Now define the process $\oper{EG}(G)$ by letting 
$$\oper{EG}(G)(x, r) = \oper{EG}_t(G)(x, r), $$
where $t$ is any time greater than $r$. This is well-defined since for $r < s < t$, $\oper{EG}_t(G)(x, r) = \oper{EG}_s(G)(x, r)$.

\medskip

It is easy to see that $\oper{EG}$ is continuous on $\scrI$, by checking that $\oper{EG}_t$ is continuous for all $t$. This is clear since if $G_n \to G$ in $\scrI$, for any subset $T(\ell_i, k_i) \sset   \half$, eventually $G_n$ will be identically ordered to $G$ on $T(\ell_i, k_i)$ and so the ordering of the swaps given by the Edelman-Greene bijection will be the same for $G_n$ and $G$ on $T(\ell_i, k_i)$. Moreover, the times at which these swaps occur converge in the limit. This implies convergence of both the cadlag paths of the permutation $\oper{EG}(G_n)(\cdot, x)$ and the cadlag paths of the inverse permutation, thus showing that $\oper{EG}$ is continuous.

\medskip

Finally, by Corollary \ref{cor-spaces} and Proposition \ref{inject}, we know that our tableau process limit $F \in \scrI$ almost surely, so 
$$
U_n = \oper{EG}(F_n) \to U = \oper{EG} (F)
$$
 in distribution as well by the continuity of the map $\oper{EG}$. This proves convergence of random sorting networks at the center to a swap process $U$. The only thing left to do to prove Theorem \ref{T:main} when $u=0$ is to show that the limit $\oper{EG}(F)$ has time-stationary increments, as the spatial stationarity and mixing follow from the spatial stationarity and mixing of $F$.

\begin{prop}
\label{time-stat} $U$ has time-stationary increments. Namely, the distribution of the process  $(U(\cdot,s)^{-1}U(\cdot,s+t)), t\ge 0)$ does not depend on $s$.
\end{prop}


\begin{proof}
The sequence of transpositions $\{\pi_{i_1}, \ddd \pi_{i_k}\}$ in a random sorting network is equal in law to the sequence $\{\pi_{i_\ell}, \ddd \pi_{i_{\ell + k -1}}\}$.  To prove this time stationarity, note that if we remove the first swap $\pi_{i_1}$ from a sorting network, we can get another sorting network by adding the swap $\pi_{n - i_1}$ to the end of the sorting network. This result was first proved in \cite{angel2007random}.

\medskip

We use this idea to extend the process $U_n$, which only completes ${n \choose 2}$ swaps at the first ${n \choose 2}$ times of a rate-$n$ Poisson process $\Pi_n$, to a process $U_n^*$, which completes swaps at every time in $\Pi_n$. Let the first $N$ swaps in $U_n^*$ be as in $U_n$ and then recursively define the $k$th swap in $U_n^*$ to be equal to $\pi_{n - j}$, where $\pi_{j}$ is the $\left(k - {n \choose 2}\right)$th swap in $U_n^*$ for $k > N$. Then $U_n^*$ is a time-stationary process, and $U_n^*(x, t) = U_n (x, t)$ for all $t \leq T_n$, where $T_n$ is the ${n \choose 2}$th point in $\Pi_n$.

\medskip

 Since $T_n \cvgd \infty$ as $n \to \infty$, and $U_n \cvgd U$, $U_n^* \cvgd U$ as well. Finally, since each $U_n^*$ has stationary increments, $U$ must have stationary increments as well. 
\end{proof}

Putting this all together, we obtain Theorem \ref{T:main} in the $u=0$ case.

\begin{theorem} 
\label{k0}
Let $a_n$ be a sequence of integers with $a_n = o(n)$. Let $U_n$ be the swap process defined by 
$$
U_n(x, t) = \sig^n_{\lfloor nt \rfloor}(a_n + x) - a_n,
$$
where $\sig^n$ is an $n$-element random sorting network. Then 
$$U_n \cvgd U,$$
where $U$ is a swap process that is stationarity and mixing of all orders with respect to the spatial shift, and has time-stationary increments.
\end{theorem}

\end{section}
\begin{section}{The Local Limit Outside the Center}
\label{S:outside}
In this section, we prove that the local limit of random reverse standard staircase Young tableaux exists at distance $\lfloor un \rfloor+ o(n)$ outside the center. This will immediately imply the existence of the local limit outside the center for sorting networks via the Edelman-Greene map $\oper{EG}$ in Section \ref{S:llcentre}. 

\begin{theorem}
\label{lim-outside}
Let $u \in (-1, 1)$, $a_n = o(n)$, and let $G_n$ be the inclusion functions for the sequence of tableau processes $X_n(\lfloor un \rfloor  + a_n, n, n)$. Then
 $$G_n \cvgd  F^u = \frac{1}{\sqrt{1 - u^2}}F,$$
 where $F$ is the limit when $u = 0$.
\end{theorem}

We will assume that $a_n = 0$ throughout, as it is easy to use domination lemmas to conclude Theorem \ref{lim-outside} for general $a_n$ from this case. The basic idea of the proof is as follows. 
By using the domination lemmas in Section \ref{S:dom}, it is easy to see that any subsequential limit $G$ at a distance $\lfloor un \rfloor$ outside the center must be stochastically dominated by $F^u$, so we just need to show domination in the opposite direction. For this, we show that the expected heights in the tableau process corresponding to $F^u$ are greater than expected heights in the tableau process corresponding to $G$ at every location and every time.

\medskip

 Note that it is possible to get domination in the opposite direction for almost every value of $u$ by comparing the number of squares in a tableau process at time $t$ with the expected number of squares in each of processes shifted by $u$, integrated over all $u \in (-1, 1)$. However, this approach only proves Theorem \ref{lim-outside} for almost every $u$. To prove the theorem for any $u$, we take the following approach. 
 
 \medskip
 
 By considering the inclusion functions $G_n$ of the shifted tableau processes as elements of $\scrH = [0, \infty]^\half$ we have a set $\scrG$ of subsequential limits of $G_n$ by compactness. Consider largest and smallest elements in $\scrG$ in the stochastic ordering on inclusion functions. Such elements exist since $\scrG$ is closed and the space of probability measures on $\scrH$ is compact. Call $G \in \scrG$ a limsup if for any $G' \in \scrG$, $G' \sgeq G$ if and only if $G' = G$. Similarly, we define a liminf in $\scrG$ to be any $G \in \scrG$ such that for $G' \in \scrG$, $G' \sleq G$ if and only if $G' = G$. 
 
 \medskip
 
 We show that these elements are translation invariant, and that any translation invariant element of $\scrG$ has expected heights less than those of $F^u$. Therefore any limsup or liminf in $\scrG$ must be $F^u$. As any element in $\scrG$ must lie between a liminf and a limsup, this allows us to conclude that $\scrG = \{ F^u \}$.
 \bigskip

{\bf Shifted tableau processes}. \qquad We introduce new notation for the tableau processes used in this section, using $Y$ instead of $X$ to distinguish from centered tableau processes. For a fixed value of $u \in (-1, 1)$, define $Y^K_n(t)$ to be the rate $Kn$ tableau process on the diagram $T(\lfloor un \rfloor, n)$.  When $K = 1$, we omit the superscript. To establish the translation invariance of liminfs and limsups, we need a basic domination lemma involving these processes.

\begin{lemma}
\label{easy-dom} Fix $u \in (-1, 1)$, and choose $\ell$ so that for every $n$, \begin{align*}
&T(\lfloor un \rfloor, n) \sset T(\lfloor u(n +\ell) \rfloor+ 2, n + \ell),\qquad \mathand \\
&T(\lfloor un \rfloor + 2, n) \sset T(\lfloor u(n+\ell) \rfloor, n + \ell).
\end{align*}
Let $T_n$ be the time when ${n \choose 2}/2$ squares have been added to $Y_n$, and let $\theta_n = \frac{n+4\ell}{n - 1}$. Then for all large enough $n$,
\begin{align}
\begin{split}
\label{tab-dom}
&Y_{n+2\ell}^{\theta_n^2} \sgeq \tau Y_{n+\ell}^{\theta_n} \sgeq Y_n \qquad \mathand \\
&Y_{n+2\ell}^{\theta_n^2} \sgeq Y_{n+\ell}^{\theta_n}  \sgeq Y_{n},
\end{split}
\end{align}

where all stochastic domination holds up to time $T_n$.
 \end{lemma}

As before, $\tau$ is the spatial shift. Thus $\tau Y^K_{n}(t)$ is exactly $Y^K_n$ shifted by 2 units to the right so that it lives on the diagram $T(\lfloor un \rfloor + 2, n)$. The essence of this lemma is that we can get domination of the shifted process $\tau Y_{n_1}$ over $Y_n$ by letting $n_1$ be slightly larger than $n$, and slightly speeding up $\tau Y_{n_1}$. The precise value of the speed-up $\theta_n$ is not important here, only that $\theta_n \to 1$ as $n \to \infty$.

 \begin{proof}
 
We just prove that $\tau Y_{n+\ell}^{\theta_n}\sgeq Y_n$ up to time $T_n$, as the rest of the inequalities follow using the same argument. By Lemmas \ref{stoch-dom-criterion} and \ref{stoch-dom-1}, we just need to show that if $Y_n$ and $\tau Y_{n+\ell}^{\theta_n}$ are in the same configuration $A$, and $z$ is a corner of both $T(\lfloor un \rfloor, n) \smin A$ and $T(\lfloor u(n + \ell) \rfloor + 2, n+\ell) \smin A$, that $v_n(z, A) < v_{n+\ell}(z, A)$, where $v_n$ and $v_{n+\ell}$ refer to rates in $Y_n$ and $Y_{n+\ell}^{\theta_n}$, respectively. To see this, observe that for any set $A$ of cardinality at most ${n \choose 2}/2$,
\begin{align*}
\frac{v_{n+\ell}(z, A)}{v_{n}(z, A)} &= \theta_n \frac{n+\ell}{n}  \frac{\card{T(\lfloor un \rfloor, n) \setminus A}}{{\card{T(\lfloor u(n + \ell) \rfloor + 2, n + \ell) \setminus A}}}  \prod_{y \in R_z^{n+\ell} \setminus R_z^n}\left(1 + \frac{1}{h_y^{n+\ell} -1} \right) \\
& \geq \frac{(n+ 4\ell)(n+\ell)}{n(n-1)} \frac{n(n-1)}{2(n+\ell)(n+\ell-1) - n(n-1)} \\
& > 1. \qedhere
\end{align*}
 \end{proof}
 Now we can characterize liminfs and limsups in $\scrG$.

 \begin{prop}
 \label{sup-trans}  Suppose $G \in \scrG$ is a limsup (or a liminf). Then $G$ is translation invariant.
 \end{prop}

 \begin{proof}
Throughout this proof, we let $G^K_{n}$ be the inclusion function of $Y^K_{n}(t)$. Let $G_{n(i)} \to G$ for some liminf $G \in \scrG$ (the case for $G$ a limsup is similar). Note that by Lemma \ref{easy-dom}, $ G' \sleq G$ for any subsequential limit $G'$ of $G^{\theta_{n(i)}^2}_{n(i) + 2\ell}$.  By passing to the limit, we remove any issues with the stopping time $T_n$ from Lemma \ref{easy-dom} since $T_n \cvgd \infty$ as $n \to \infty$. Such limits exist by compactness of $\scrH$.

\medskip

However, since $\theta_n^2 \to 1$ as $n \to \infty$, $G'$ is also a subsequential limit of $G_{n(i) + 2\ell}$, so since $G$ is a liminf, $G' = G$. Therefore $G^{\theta_{n(i)}^2}_{n(i) + 2\ell} \cvgd G$. Now again by Lemma \ref{easy-dom}, we have that 
\begin{align*}
&\close{G}^{\theta_{n(i)}^2}_{n(i) + 2\ell} \sleq \close{G}^{\theta_{n(i)}}_{n(i) + \ell} \sleq \close{G}_{n(i)} \qquad \mathand \\
&\close{G}^{\theta_{n(i)}^2}_{n(i) + 2\ell} \sleq \close{G}^{\theta_{n(i)}}_{n(i) + \ell} \circ \tau \sleq \close{G}_{n(i)},
\end{align*}
where $\close{G_*} = G \wedge T_{n}$ for each of the inclusion functions $G_*$ corresponding to the tableau processes in \eqref{tab-dom}. Note here that if $G$ is the inclusion function for the process $Y$, then $G \circ \tau$ is the inclusion function for the shifted process $\tau Y$. By the squeeze theorem, and the facts that $\theta_n \to 1$ and $T_n \cvgd \infty$, this implies that both $G_{n(i) + \ell} \cvgd G$ and $G_{n(i) + \ell} \circ \tau \cvgd G$, allowing us to conclude that $G  \circ \tau \eqd G$.
 \end{proof}

 We now aim to show that every translation-invariant element $G \in \scrG$ is the rescaled central limit $F^u$ by comparing heights. For any $J \in \scrH$, $x \in 2\inte + 1$ and $t \in [0, \infty)$, define the height function

 $$h(J, x, t) = \card{ \{z = (z_1, z_2): z_1 = x \mathor x + 1 \mathand J(z) < t \} }.$$

 We first prove the following lemmas about the expected heights in $F$.
\begin{lemma}
\label{finite-expect}
$\expt h(F, x, t)$ is finite for all $t \in [0, \infty)$ and $x \in 2 \inte$.
\end{lemma}

\begin{proof}
Note that $F_n \cvgd F$, and that
$$h(F_n, x, t) \cvgd h(F, x, t)$$
for all $t$ and $x$ since $F$ has no atoms. Recall also that the tableau processes $X_n$ are dominated by a sped-up cylinder process $C(n, 8n)$ up to the stopping time $T_n$ when $n^2/4$ squares have been to $X_n$. 
Since $T_n\to \infty$ in probability as $n\to \infty$, we also have 
$$h(F_n ,x, t\wedge T_n) \cvgd h(F, x, t).$$
By the symmetry of the cylinder, the expected height at $x$ at time $t$ in $C(n, 8 n)$ is $8t$, so by Fatou's lemma,

\[
\expt h(F, x, t) \leq \liminf_{n \to \infty} \expt h(F_n, x, t\wedge T_n) \leq 8 t. \qedhere
\]

\end{proof}
 \begin{lemma}
 \label{dct-lem} Let $T^n_t$ to be the stopping time when $\lfloor nt \rfloor$ squares have been added to the centered tableau process $X_n = X(0, n, n)$. There exists a subsequence $\{n_i : i \in \nat\}$ such that 
 $$\expt h(F_{n_i}, x, T^{n_i}_t) \to \expt h(F, x, t).$$

\end{lemma}

 \begin{proof}
 
We find a dominating ``infinite tableau process" for the sequence of tableau processes $X_n$. We can find an increasing sequence $\{n_i : i \in \nat\}$ and a decreasing sequence $\{\delta_i : i \in \nat\}$ such that for all $i$, the tableau process 
$$Z_i = X(0, n_i, (1 + \delta_i)n_i)$$
 stochastically dominates the process $Z_{i-1}$ up to time $T^{n_{i-1}}_t$, and such that
$$\prod_{i=1}^\infty(1 + \delta_i) < \infty.$$
Finding such sequences can easily be done by iteratively choosing $n_1, n_2$ and $\ep$ appropriately in Lemma \ref{stoch-dom-general} (noting that that domination in that lemma is up to the time when $\ep n_1^2$ squares have been added, so we can let $\ep$ become arbitrarily small for large $n_1$ and still have domination up to time $T^{n_1}_t$).
Then letting $J_i$ be the inclusion function for $Z_i$, we have
$$
J_i = \prod_{j=1}^i(1 + \delta_j)^{-1}F_{n_i} \cvgd J = \prod_{j=1}^\infty (1 + \delta_j)^{-1} F.
$$ $J_i$ is a monotone decreasing sequence in the stochastic ordering. Moreover,
$
F_{n_i} \sgeq J_i \sgeq J$  so $h(F_{n_i}, x, t)  \sleq h(J, x, t),
$
for every $x$ and $t$. Finally, heights in $J$ have finite expectation by Lemma \ref{finite-expect} as $J$ is a sped-up version of $F$. Therefore the dominated convergence theorem, 
\[
\expt h(F_{n_i}, x, t) \to \expt h(F, x, t).
\]
As $| \expt h(F_{n_i}, x, t) - \expt h(F_{n_i}, x, T^{n_i}_t)| \to 0$ as $n \to \infty$, this completes the proof.
\end{proof}
  In order to compare the heights in $F^u$ and $G$ we will need to translate the tableau processes to swap processes on the integers. The reason for doing this is that we can relate the expected height at position $x$ to the expected number of swaps at position $x$, and the expected number of swaps at any position in a sorting network is given by the following theorem from \cite{angel2007random}.
  \begin{theorem}
  \label{semi}
  Let $\sig$ be a random sorting network on $n$ particles given by a sequence of adjacent transpositions $\{\pi_{k_1}, \ddd \pi_{k_N}\}$, and let $a_n$ be a sequence of positive integers with $\frac{2a_n}{n} - 1\to u \in (-1, 1)$. Then
  \begin{align*}
 \qquad \qquad n\prob(k_1 = a_n) \to \frac{4}{\pi}\sqrt{1 - u^2} \qquad \mathand \qquad \expt\big(\card{ \{i \leq Cn : k_i = a_n\}}\big) \to \frac{4C}{\pi}\sqrt{1 - u^2}.
  \end{align*}
    \end{theorem}
    
\noindent We use this theorem to prove the following lemma about expected height in $F$.

\begin{lemma}
\label{height-lim}
$$\lim_{t \to 0} \frac{\expt h(F, 0, t)}{t} \geq \frac{4}{\pi}.$$
\end{lemma}

\begin{proof}
By Lemma \ref{dct-lem}, we can first replace $\expt h(F, 0, t)$ by $\lim_{n \to \infty} \expt h(F_{n_i}, 0, T^{n_i}_t)$. Now we replace $h(F_{n_i}, 0, t)$ by the strictly smaller quantity $\indic(F_{n_i}(z_0) < T^{n_i}_t)$ where $z_0 = (1,1)$, and note that
$$
\prob(F_{n_i}(z_0) < T^{n_i}_t) \geq  1 - \left(1 - \frac{p_i}{n_i}\right)^{\lfloor n_it \rfloor} 
$$
where $p_i = v_{X_{n_i}} (z_0, \emptyset)$. We can make this replacement since the rate of adding the square $z_0$ is monotone increasing in time. Now by Theorem \ref{semi}, $p_i \to \frac{4}{\pi}$ as $i \to \infty$, so we have
\begin{align*}
\lim_{t \to 0} \frac{\expt h(F, 0, t)}{t} \geq \lim_{t \to 0} \frac{1 - e^{-\frac{4t}{\pi}}}{t} = \frac{4}{\pi},
\end{align*}
as desired.
\end{proof}
  For $x \in 2\inte + 1$, we now define $s(J, x, t)$ to be the number of swaps at location $x$ before time $t$ in the swap process $\oper{EG}(J)$, where the map $\oper{EG}$ is as in Section \ref{S:llcentre}.  We then have the following relationships between heights and swaps.

 \begin{lemma}
 \label{swap-height}
 Let $x \in 2 \inte$, $t \in [0, \infty)$, and let $G \in \scrG$ be translation invariant. Then $G \in \scrI$ and $\expt h(G, x, t) = \expt s(G, x, t)$ ($\scrI$ is defined at the beginning of Section \ref{S:llcentre}). We also have that $\expt h(F, x, t) = \expt s(F, x, t)$. 
 \end{lemma}

 \begin{proof}
We can use the bound in Lemma \ref{stoch-dom-general} to conclude that $G \sgeq F$, thus implying that at any time $t$, there is a bi-infinite sequence of squares in the bottom row that have not been added to $G$. Moreover, there exists a constant $C$ such that for all large enough $n$ the modified rates in each $G_n$ are bounded up to the stopping time $T$ when $n^2/4$ squares have been added by $C$ times the modified rate in $F_n$. This allows us to conclude that $G$ is injective, by the proof of Proposition \ref{inject}. Therefore $G \in \scrI$.

\medskip

Thus we can apply the Edelman-Greene map $\oper{EG}$ from Section \ref{S:llcentre} to $G$, giving a translation-invariant swap process $\oper{EG}(G)$ and allowing us to define $s(G, x, t)$ for all $t$. We now show that $\expt h(G, x, t) = \expt s(G, x, t)$. By translation invariance, it suffices to consider the case $x = 1$. For each square $z \in \half$, let $\pi(z) \in 2 \inte + 1$ be the location of the swap in $\oper{EG}(G)$ corresponding to the square $z$. Since only squares $z' \geq z_0$ can have $\pi(z) = 1$, we have

\begin{align}
\begin{split}
\label{sum-s-h}
\expt s(G, 1, t) &= \sum_{z' \geq z_0} \prob(F(z') < t \mathand \pi(z') = 1) \\
&= \sum_{i=1}^\infty \sum_{j \in [1 - (i-1),  1 + (i-1)]} \prob(F(j, i) < t \mathand \pi(j, i) = 1) \\
&= \sum_{i=1}^\infty \sum_{j \in [1 - (i-1),  1 + (i-1)]} \prob\big(F(q_i, i) < t \mathand \pi(q_i, i) = 1+ q_i - j\big). \\
\end{split}
\end{align}
Here $q_i$ is either $1$ or $2$ depending on the parity of $i$. The second equality is just rearranging terms in the sum and the final equality comes from the translation invariance of the swap process. Since $$\pi(q_i, i) \in [q_i - (i-1),  q_i + (i-1)],$$ we have
$$
\sum_{j \in [- i , i ]} \prob\big(F(q_i, i) < t \mathand \pi(q_i, i) = 1 - j\big) = \prob(F(q_i, i) < t),
$$
and so the final line of \eqref{sum-s-h} is equal to $\expt h(G, 1, t)$. The exact same proof works for $F$.

\end{proof}
 \begin{prop}
 \label{trans-is-rescaled}
 Suppose $G \in \scrG$ is translation invariant. Then $G \eqd \frac{1}{\sqrt{1 - u^2}}F$.
 \end{prop}

 \begin{proof}
 First define 
 $$\scrK = \{f: 2\Z + 1 \times \real_+ \to \inte\},$$
and define  $H: \scrH \to \scrK$ by $H(J) = h(J, \cdot, \cdot)$. Note that $H$ a strictly decreasing function with respect to the pointwise orders on  $\scrH$ and $\scrK$. As every $G \in \scrG$ satisfies $G \sleq F^u$, to show that $F^u \eqd G$ it suffices to show that $\expt h(G, x, t) \leq \expt h(F^u, x, t)$ for all $x$ and $t$. By Theorem \ref{semi},
 $$\expt s(G_n, 0, t)  \to \left(\frac{4}{\pi} \sqrt{1 - u^2}\right)t.$$
Then by Fatou's Lemma and Lemma \ref{swap-height} we have that
 \begin{align}
 \label{sG}
 \expt h(G, 0, t) = \expt s(G, 0, t) \leq \lim_{n \to \infty} \expt s(G_n, 0, t) &=  \left(\frac{4}{\pi}\sqrt{1 - u^2}\right) t.
 \end{align}

Now by the time-stationarity of the increments in the limit EG$(F)$ (Proposition \ref{time-stat}), we have that $\expt s(F, 0, t)$ is linear in time. Therefore $\expt h(F, 0, t)$ must be linear in time as well since it is equal to 
$\expt s(F, 0, t)$ by Lemma \ref{swap-height}. Combining this with Lemma \ref{height-lim} gives that $\expt h(F, 0, t) = K t$ for some $K \geq \frac{4}{\pi}$, so

$$\expt h( F^u, 0, t) \geq  \left(\frac{4}{\pi}\sqrt{1 - u^2} \right)t,$$
which combined with \eqref{sG} gives the desired result.

\end{proof}

\begin{proof}[Proof of Theorem \ref{lim-outside}.]
We can finally combine Propositions \ref{sup-trans} and
\ref{trans-is-rescaled} to conclude the convergence of the processes
$G_n$ to $F^u$, which completes the proof of Theorem
\ref{lim-outside}. This in turn completes the proofs of
Theorems \ref{T:main} and \ref{T:YT_limit}.
\end{proof}

Proposition \ref{trans-is-rescaled} also allows us to conclude the following proposition about expected heights in $F$, and therefore swaps in $EG(F)$.

\begin{prop} For any $x$ and $t$, we have
$$\expt h(F, x, t) = \expt s(F, x, t) = \frac{4}{\pi}t.$$
\end{prop}

\end{section}

\begin{section}{Appendix}
\label{S:appendix}
\begin{proof}[Proof of Lemma \ref{domination-implies-convergence}.]
$G_n$ is tight, so it has subsequential limits in distribution. Suppose that $G^1$ and $G^2$ are two different subsequential limits of $G$. Then there are subsequences $G_{\al(i)} \cvgd G^1$ and $G_{\beta(i)} \cvgd G^2$. Without loss of generality, we can assume that there are some numbers $a_1, \ddd a_m  > 0$ such that

\begin{equation*}
\prob\left(G^1 \in \prod_{k=1}^m[0, a_k]\right) - \prob\left(G^2 \in \prod_{k=1}^m[0, a_k]\right) > 0.
\end{equation*}
Then there is some $\delta > 0$ such that

\begin{equation*}
\prob\left(G^1 \in \prod_{k=1}^m[0, a_k+ \delta)\right) - \prob\left(G^2 \in  \prod_{i=1}^m[0, a_k + 2 \delta]\right) > 0,
\end{equation*}
By weak convergence, we get the following chain of inequalities.
\begin{align}
\begin{split}
\label{eqn-cvg-almost}
\limsup_{i \to \infty} \prob\left(G_{\beta(i)} \in \prod_{k=1}^m[0, a_k + 2 \delta]\right) &\leq \prob\left(G^2 \in \prod_{k=1}^m[0, a_k + 2 \delta]\right) \\
&< \prob\left(G^1 \in \prod_{k=1}^m[0, a_k + \delta)\right) \\
&\leq \liminf_{i \to \infty} \prob\left(G_{\al(i)} \in \prod_{k=1}^m[0, a_k + \delta)\right).
\end{split}
\end{align}
 However, letting $\ep = \frac{\delta}{a + \delta}$ where $a = \max_k a_k$, for any large enough $i$ there exists some $J$ such that for all $j \geq J$,

 \begin{align*}
 \prob\left( G^\ep_{\al(i)} \in \prod_{k=1}^m[0, a_k + \delta)\right) &\leq \prob\left((1 + \ep) G^\ep_{\al(i)} \in \prod_{k=1}^m[0, a_k + 2\delta]\right) \\
 &\leq \prob\left(G_{\beta(j)} \in \prod_{k=1}^m[0, a_k + 2\delta] \right),
 \end{align*}
since $(1+\ep)G^\ep_{\al(i)} \sgeq G_{\beta(j)}$ for all large enough $j$ by assumption. Thus

 \begin{equation*}
 \limsup_{i \to \infty}  \prob\left( G_{\beta(i)} \in \prod_{k=1}^m[0, a_k + 2\delta] \right) \geq \liminf_{i \to \infty} \prob\left(G_{\al(i)}^\ep \in \prod_{k=1}^m[0, a_k + \delta)\right),
 \end{equation*}
which contradicts \eqref{eqn-cvg-almost}, since
 \begin{align*}
 \liminf_{i \to \infty} \prob\left(G^\ep_{\al(i)} \in \prod_{k=1}^m[0, a_k + \delta)\right) &= \liminf_{i \to \infty} \prob\left(G_{\al(i)} \in \prod_{k=1}^m[0, a_k + \delta)\right).
 \end{align*}
  Thus $G^1 = G^2$ for any two subsequential limits of $G_n$, so $G_n$ has a distributional limit.
\end{proof}

\end{section}

{\noindent \bf Acknowledgements.} Omer Angel was supported in part by NSERC. Duncan Dauvergne was supported by an NSERC CGS D scholarship. B\'alint Vir\'ag was supported by the 
Canada Research Chair program, the NSERC Discovery Accelerator
grant, the MTA Momentum Random Spectra research group, and the ERC 
consolidator grant 648017 (Abert). We would also like to thank the Banff International Research Station for hosting a focussed research group that initiated this research.

\bibliography{LLRSNbib2}
\end{section}

\end{document}